\documentclass[a4paper]{amsart}

\usepackage[T1]{fontenc}
\usepackage[utf8]{inputenc}
\usepackage[english]{babel}
\usepackage{lmodern}
\usepackage{exscale}
\usepackage[babel]{microtype}
\usepackage{amsmath, amssymb, mathtools, mathrsfs}

\usepackage{tikz}
\usetikzlibrary{matrix,shapes,arrows,calc,3d,decorations,decorations.pathmorphing,through,cd}
\usepackage{todonotes}

\usepackage[numbered]{bookmark}
\usepackage{hyperref}
\usepackage{amsthm}
\hypersetup{
  colorlinks,
  urlcolor = {blue!80!black},
  linkcolor = {red!70!black},
  citecolor = {green!50!black}
}

\theoremstyle{plain}
  \newtheorem{thm}{Theorem}
  \newtheorem{defi}[thm]{Definition}
\newtheorem{conj}[thm]{Conjecture}
  \newtheorem{prop}[thm]{Proposition}
  
  \newtheorem{cor}[thm]{Corollary}
  
  \newtheorem{lemma}[thm]{Lemma}
\theoremstyle{definition}
  \newtheorem{ex}[thm]{Example}
  \newtheorem{rem}[thm]{Remark}

  \usepackage{float}

\usepackage{diagbox}
\usepackage{color, colortbl}
\usepackage{array}
\usepackage{varwidth} 

\definecolor{Gray}{gray}{0.9}
\definecolor{LightGray}{gray}{0.95}

\newcolumntype{g}{>{\columncolor{Gray}}c}
\newcolumntype{M}{V{3cm}}
\newcolumntype{D}{V{4cm}}
\newcolumntype{H}{V{4cm}}
\newcolumntype{F}{p{.45cm}} 

\tikzset{ext/.style={circle, draw,inner sep=1pt},int/.style={circle,draw,fill,inner sep=1pt},nil/.style={inner sep=1pt}}
\tikzset{exte/.style={circle, draw,inner sep=3pt},inte/.style={circle,draw,fill,inner sep=3pt}}
\tikzset{diagram/.style={matrix of math nodes, row sep=3em, column sep=2.5em, text height=1.5ex, text depth=0.25ex}}
\tikzset{diagram2/.style={matrix of math nodes, row sep=0.5em, column sep=0.5em, text height=1.5ex, text depth=0.25ex}}
\tikzset{every picture/.style={baseline=-.65ex}}

\newcommand{\R}{{\mathbb{R}}}
\newcommand{\Q}{\mathbb{Q}}

\newcommand{\mF}{\mathcal{F}}

\newcommand{\Com}{\mathsf{Com}}
\newcommand{\bCom}{\mathbf B\Com}

\newcommand{\bpm}{\begin{pmatrix}}
\newcommand{\epm}{\end{pmatrix}}

\newcommand{\rank}{\mathrm{rank}\,}

\newcommand{\G}{\mathrm{G}}
\newcommand{\GMe}{\G^{\mathrm{Me}}}

\newcommand{\HG}{\mathsf{HG}}
\newcommand{\CHG}{\mathsf{CHG}}

\DeclareMathOperator{\sgn}{sgn}
\DeclareMathOperator{\gr}{gr}

\newcommand{\beq}[1]{\begin{equation}\label{#1} 
}
\newcommand{\eeq}{\end{equation}}

\DeclareMathOperator{\vimg}{\mathrm{im}}

\newcommand{\FG}{\mathsf{FG}}
\newcommand{\FGC}{\mathsf{FGC}}

\newcommand{\Out}{\mathrm{Out}}
\newcommand{\lp}{\text{-loop}}
\newcommand{\exc}{\text{-exc}}

\newcommand{\Aut}{\mathrm{Aut}}
\newcommand{\Det}{\mathrm{Det}}

\newcommand{\GH}{GH }
\usetikzlibrary{decorations,decorations.pathmorphing,decorations.pathreplacing,decorations.markings}
\tikzset{fatedge/.style={line width=.5mm}}
\tikzset{pbrace/.style={decorate, decoration = {brace}}}

\DeclareMathOperator{\mim}{\mathrm{im}}

\renewcommand{\FGC}{\FG}

\usepackage{fullpage}
\usepackage{float}

\usepackage{diagbox}
\usepackage{color, colortbl}
\usepackage{array}
\usepackage{varwidth} 

\definecolor{Gray}{gray}{0.9}
\definecolor{LightGray}{gray}{0.95}

\usepackage{csquotes}
\usepackage[style=alphabetic, sorting=nyt, maxnames=10, maxalphanames=5]{biblatex}
\begin{document}

\title{Graph homology computations}

\author[S. Brun]{Simon~Brun}
\address{Department of Mathematics, ETH Zurich, Zurich, Switzerland}
\email{simon.brun@math.ethz.ch}

\author[T. Willwacher]{Thomas~Willwacher}
\address{Department of Mathematics, ETH Zurich, Zurich, Switzerland}
\email{thomas.willwacher@math.ethz.ch}

\subjclass[2000]{16E45; 53D55; 53C15; 18G55}
\keywords{}

\begin{abstract}
  We compute numerically the homology of several graph complexes in low loop orders, extending previous results.
\end{abstract}

\maketitle

\section{Introduction}
\label{sec:introduction}
Graph complexes are differential graded vector spaces of linear combinations of isomorphism classes of graphs.
There exist many variants, using different types of graphs, for example, simple graphs, ribbon graphs, directed acyclic graphs and many more.
The differential on graph complexes is typically given by the operation of edge contraction, though again variants exist.

Graph complexes play a central role in many areas of algebraic topology and homological algebra. Most often, they constitute a combinatorial encoding of a homotopy theoretic problem.
To name a few examples, complexes of simple graphs encode automorphisms of the chains operad of the little disks operad \cite{Fresse2017b, Willwacher2014}, hairy graphs govern the rational homotopy type of higher dimensional knot spaces \cite{FresseTurchinWillwacher2017} and the ribbon graph homology computes the homology of moduli spaces of curves \cite{Penner88}.

In most cases, little is known about the homology of the complexes, the graph homology, which has a rich and complicated structure, inherited from the homotopy theoretic problems the graph complex encodes.
At the same time it is relatively hard for a mathematician working in the area to come up with, or rule out, conjectures or statement about that structure. This is due mostly to the quickly growing combinatorial complexity of graphs, which makes it difficult to conduct computations by hand.
Furthermore, there is typically a significant time cost associated to implementing such computations on a computer, paired with a non-zero probability of eventually obtaining incorrect answers due to programming mistakes.

Here we report on computer experiments, determining the dimensions of the graph homology in low degrees for several common graph complexes. We furthermore test and verify or support numerically some existing results and conjectures on the graph homology, at least in low loop orders.
We also present a new conjecture (Conjecture \ref{conj:morita new} below) about the forested graph homology, inspired by our computational results.
The reader interested in the numerical results about graph complexes only, and not in preliminaries, is invited to jump directly to the data tables in section \ref{sec:experiments}.
Mind however that a significant part of the computations has been performed over finite fields for memory reasons -- for a discussion about the validity over $\Q$ please consult the "Methods" subsections within section \ref{sec:experiments}.

To conduct the experiments, we wrote the \GH framework and library, discussed briefly in section \ref{sec:GH} below. \GH contains classes and routines for handling the generation of bases of graph complexes, implementing linear operators on such complexes, managing data files produced, and visualising results.
The \GH library is an attempt to unify numerical experiments on graph complexes, which have appeared scattered throughout the literature \cite{BarNatanMcKay, Murri2012,BuringKiselevRutten2017}, and to enable other reasearchers to implement similar experiments relatively quickly.

\subsection*{Acknowledgements}
We are greatly indebted to the IT support group of the mathematics department of ETH Zurich, who helped significantly to get our code to run on the departmental servers.

\section{Graph complexes}

In this section we introduce the graph complexes whose homology we study numerically in this paper.
All our graph complexes will be differential graded vector spaces over a fixed ground ring $R$, that is usually $R=\mathbb Q$. We use homological conventions, that is, our differentials have degree -1.

\subsection{The Kontsevich graph complex}
\newcommand{\ori}{\mathit{or}}
Possibly the simplest graph complexes are the Kontsevich graph complexes $\G_n$, for $n$ an integer.

We define an \emph{admissible graph} to be a connected
1-vertex irreducible simple graph all of whose vertices have valence $\geq 3$.
Here, simple means that the graph is undirected without multiple edges or edges connecting a vertex to itself. One-vertex irreducibility means that deleting any single vertex leaves the graph connected.

An orientation of an admissible graph $\gamma$ are the following data, depending on the parity of $n$.
\begin{itemize}
\item For $n$ even, an orientation is an ordering of the set of edges of $\gamma$, up to even permutations.
We will write 
$$\ori=e_1\wedge e_2\wedge \cdots \wedge e_k$$ 
for the ordering $e_1>e_2 >\cdots > e_k$ of the set of edges $E\Gamma=\{e_1,\dots,e_k\}$ of $\gamma$.
\item For $n$ odd, an orientation is an ordering of the set of vertices and half-edges, up to even permutations. Let the set of vertices be $V\gamma=\{v_1,\dots,v_p\}$ and the set of half-edges be $H\gamma=\{h_1,\dots,h_{2k}\}$. Then the orientation is an ordering of the set $H\gamma\sqcup V\gamma$ and, as above, we denote such an orientation by a wedge product as follows:
$$\ori = h_1\wedge \cdots \wedge h_{2k} \wedge v_1\wedge \cdots \wedge v_p.$$ 
\end{itemize}
In any case, there are only two different orientations for an admissible graph. Given some orientation $\ori$, we denote by $-\ori$ the other permutation.
An oriented graph is a pair $(\gamma,\ori)$ of an admissible graph $\gamma$ and an orientation $\ori$ of $\gamma$.
Two oriented graphs $(\gamma,\ori)$ and $(\gamma',\ori')$ are called isomorphic if there is an isomorphism between $\gamma$ and $\gamma'$ that maps $\ori$ to $\ori'$.

Elements of the Kontsevich graph complex $\G_n$ are then $R$-linear combinations of oriented graphs, modulo the following relations.
\begin{itemize}
\item We identify isomorphic oriented graphs.
\item We identify opposite orientations up to sign, $(\gamma,\ori)=-(\gamma,-\ori)$.
\end{itemize}
For example, for $n$ even this means that for any permutation $\sigma\in S_k$
\[
(\gamma, e_1\wedge e_2\wedge \cdots \wedge e_k) 
=(-1)^\sigma  (\gamma, e_{\sigma(1)}\wedge  \cdots \wedge e_{\sigma(k)}), 
\]
thus justifying the wedge product notation.
For illustration purposes, we depict elements of $\G_n$ as linear combinations of graphs, for example
  \[
    \begin{tikzpicture}[baseline=-.65ex, scale=.5]
\node[int] (c) at (0,0){};
\node[int] (v1) at (0:1) {};
\node[int] (v2) at (72:1) {};
\node[int] (v3) at (144:1) {};
\node[int] (v4) at (216:1) {};
\node[int] (v5) at (-72:1) {};
\draw (v1) edge (v2) edge (v5) (v3) edge (v2) edge (v4) (v4) edge (v5)
      (c) edge (v1) edge (v2) edge (v3) edge (v4) (c) edge (v5);
\end{tikzpicture}
+
\frac 5 2\,
    \begin{tikzpicture}[baseline=-.65ex, scale=.5]
\node[int] (c) at (0.7,0){};
\node[int] (v1) at (0,-1) {};
\node[int] (v2) at (0,1) {};
\node[int] (v3) at (2.1,-1) {};
\node[int] (v4) at (2.1,1) {};
\node[int] (d) at (1.4,0) {};
\draw (v1) edge (v2) edge (v3)  edge (d) edge (c) (v2) edge (v4) edge (c) (v4) edge (d) edge (v3) (v3) edge (d) (c) edge (d);
\end{tikzpicture} \, \in \G_n,
  \]
where we leave implicit the orientation.
The vector space $\G_n$ is naturally graded.
More concretely, we declare that a graph $\gamma$ with $v$ vertices and $e$ edges has homological degree
\[
|\gamma|  = n(v-1) -(n-1)e 
\]

We define a linear operator $d$ of degree $+1$ on $G_n$ by contracting edges.
For $n$ even we set 
\begin{equation}\label{equ:d def ordinary}
d(\gamma, e_1\wedge  \cdots \wedge e_k) = \sum_{j=1}^k (-1)^{k+1} ( \gamma / e_j, e_1\wedge  \cdots \wedge \hat e_j \wedge \cdots \wedge e_k),
\end{equation}
where the graph $\gamma/e_j$ is obtained by contracting the edge $e_j$. If the graph $\gamma/e$ is not simple, then we set the term to zero.
We also write this same formula \eqref{equ:d def ordinary} as 
\[
  d(\gamma, \ori)
  =
  \sum_{e\in E\gamma}
  (\gamma/e, \partial_e \ori),
\]
where the sum is over edges $e$ of $\gamma$, and $\partial_e \ori$ is the derivative with respect to a formal variable corresponding to the edge $e$, of the orientation, considered as a monomial of the variables corresponding to edges.

For $n$ odd the formula is similar, but we need some notation.
For an edge $e\in E\gamma$ let us denote by $v_e^1$, $v_e^2$ the vertices connected by $e$, and by $h_e^1$, $h_e^2$ the two half-edges of $e$, in the same order. 
That is, $h_e^j$ is incident to $v_e^j$. The graph $\gamma/e$ is obtained by merging $v_e^1$ and $v_e^2$ to one new vertex, that we call $v_e$.
\[
  \gamma =
\begin{tikzpicture}
  \node[int, label=90:{$ v_e^1$}] (v) at (0,0) {};
  \node[int, label=90:{$ v_e^2$}] (w) at (2,0) {};
  \draw (v) edge node[above] {$e$} node[below, pos=.25] {$\scriptstyle h_e^1$} node[below, pos=.75] {$\scriptstyle h_e^2$} (w) edge +(-.5,.5) edge +(-.5,-.5) edge +(-.5,0)
  (w) edge +(.5,.5) edge +(.5,-.5) edge +(.5,0);
\end{tikzpicture}
\xrightarrow{\text{contract}}
\begin{tikzpicture}
  \node[int, label=90:{$v$}] (v) at (0,0) {};
  \draw (v) edge +(-.5,.5) edge +(-.5,-.5) edge +(-.5,0)
  edge +(.5,.5) edge +(.5,-.5) edge +(.5,0);
\end{tikzpicture}
=\gamma/e
\]

Then we define, for odd $n$,
\begin{equation}\label{equ:d def ordinary 2}
  d(\gamma, \ori) = \sum_{e\in E\gamma}
  (\gamma/e, v_e\wedge ( \partial_{v_e^1}\partial_{v_e^2}\partial_{h_e^1}\partial_{h_e^2} \ori)),
\end{equation}
using analogous partial derivative notation as above.

It is an exercise to check that $d$ is actually well-defined and that $d^2=0$.
The Kontsevich graph homology is then 
\[
H(\G_n)=\ker d/\mim d.
\]

Note that in the definition of $\G_n$ the integer $n$ enters only in the above degree convention and in the definition of orientation, and in the latter only up to parity. Hence, essentially, there are only two fundamentally distinct graph complexes defined here, $G_n$ for even $n$ and $G_n$ for odd $n$.

\begin{rem}
Our graph complex $\G_n$ is the smallest version of the Kontsevich graph complex. Analogous complexes allowing a larger class of admissible graphs have been considered in the literature, but have essentially the same homology.
For example, dropping the 1-vertex-irreducibility condition leads to a quasi-isomorphic complex \cite[Appendix F]{Willwacher2014}.
Allowing vertices of any valence leads (with the appropriate extension of the differential) to a complex that is quasi-isomorphic to $G_n$ in loop orders $>1$.
Dropping the simplicitly condition by allowing tadpoles or multiple edges yields a graph homology that is larger by at most one dimension \cite{Willwacher2014,WillwacherZivkovic15}. 
Finally, one can allow graphs to be directed with $\geq 2$-valent vertices, and also obtain a quasi-isomorphic complex \cite[Appendix K]{Willwacher2014}.
\end{rem}

\subsubsection{Merkulov's subcomplex}
The graph complex $\G_n$ comes with a descending bounded above filtration 
\[
\G_n = \mF_3 \G_n \supset \mF_4 \G_n \supset \cdots 
\]
such that $\mF_p \G_n\subset \G_n$ is spanned by graphs whose highest-valent vertex has valence $\geq p$. This filtration is compatible with the differential $d$, $d\mF_p \G_n\subset \mF_p\G_n$, since the contraction of edges can only increase valences of vertices, but not decrease them.

Let 
\[
\pi_{5} \colon \G_n \to \mF^5 \G_n
\]
be the natural projection of graded vector spaces, defined such that $\pi_5\Gamma=0$ if the graph $\Gamma$ has only vertices of valence 3 and 4, and $\pi_5\Gamma=\Gamma$ otherwise.
We may then consider Merkulov's truncated graph complex
\[
\GMe_n = \{\Gamma\in \G_n \mid \pi_5\Gamma = \pi_5 d\Gamma = 0\} \subset \G_n.
\]
The differential $d$ obviously restricts to $\GMe_n$.

\begin{conj}[Merkulov's conjecture]\label{conj:merkulov}
The inclusion of dg vector spaces $\GMe_n \subset \G_n$ induces a homology isomorphism 
\[
  H_k(\GMe_n) \cong H_k(\G_n)
\]
for all $n$ and $k$.
\end{conj}

\subsection{Hairy graph complex}
We define an \emph{admissible hairy graph} to be a connected
1-vertex irreducible simple graph all of whose vertices have valence $1$ or valence $\geq 3$.
We call the vertices of valence 1 the hairs (or legs) of the graph, and the other vertices the internal vertices.
In pictures, we only draw the internal vertices.

   
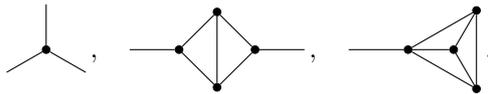
\[
\begin{tikzpicture}[scale=.6]
\node[int] (v) at (0,0){};
\draw (v) -- +(90:1) (v) -- ++(210:1) (v) -- ++(-30:1);
\end{tikzpicture}
\,,\quad
\begin{tikzpicture}[scale=.5]
\node[int] (v1) at (-1,0){};\node[int] (v2) at (0,1){};\node[int] (v3) at (1,0){};\node[int] (v4) at (0,-1){};
\draw (v1)  edge (v2) edge (v4) -- +(-1.3,0) (v2) edge (v4) (v3) edge (v2) edge (v4) -- +(1.3,0);
\end{tikzpicture}
 \, ,\quad
 \begin{tikzpicture}[scale=.6]
\node[int] (v1) at (0,0){};\node[int] (v2) at (180:1){};\node[int] (v3) at (60:1){};\node[int] (v4) at (-60:1){};
\draw (v1) edge (v2) edge (v3) edge (v4) (v2)edge (v3) edge (v4)  -- +(180:1.3) (v3)edge (v4);
\end{tikzpicture}
 \, .
 \]
 Depending on (the parity of) two fixed integers $m$ and $n$ we define the orientation of an admissible hairy graph $\gamma$ to be the following.
 \begin{itemize}
 \item For $n$ even, an orientation contains a linear ordering of the set of edges of $\gamma$, up to even permutations.
 For $n$ odd, an orientation contains an ordering of the set of vertices and half-edges, up to even permutations.
 In either case, the edges at the hairs also count as edges for the orientation.
 \item In addition, for $m$ odd the orientation contains an ordering of the set of hairs, up to even permutations.
 \end{itemize}
 
 Again, there are only two different orientations for an admissible graph.
 An oriented hairy graph is a pair $(\gamma,\ori)$ of an admissible hairy graph $\gamma$ and an orientation $\ori$ of $\gamma$.
 Two oriented hairy graphs $(\gamma,\ori)$ and $(\gamma',\ori')$ are called isomorphic if there is an isomorphism between $\gamma$ and $\gamma'$ that maps $\ori$ to $\ori'$.
 
 Elements of the hairy graph complex $\HG_{m,n}$ are then $R$-linear combinations of oriented graphs, modulo the relations of identifying isomorphic oriented graphs, and identifying $(\gamma,\ori)=-(\gamma,-\ori)$.
 The vector space $\HG_{m,n}$ is typically graded by assigning a hairy graph $\gamma$ with $v$ internal vertices, $e$ edges and $h$ hairs the homological degree 
 \[
 |\gamma| = nv+ m(h-1)-(n-1)e .  
 \] 
We may define a differential $d:\HG_{m,n}\to \HG_{m,n}$ by the same formulas \eqref{equ:d def ordinary} and  \eqref{equ:d def ordinary 2} as in the previous section, just with the caveat that we do not contract edges that are connected to a univalent (hair) vertex.

The graph complexes $\HG_{m,n}$ compute in particular the duals of the rational homotopy groups of embedding spaces $\mathbb R^m\to \R^n$, see \cite{AroneTurchin2015, FresseTurchinWillwacher2017} for the precise formulation.

\subsubsection{Variant with numbered hairs}
In the above graph complex $\HG_{m,n}$ we did not distinguish hairs in that we allowed isomorphisms of graphs to permute the hairs. One can generalize the definition and require that the hair vertices are numbered from $1$ to $r$, each number occurring exactly once, and require that graph isomorphisms must send each hair to a hair with the same number.
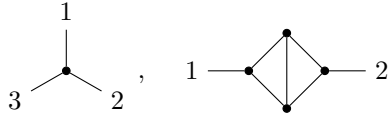
\[
\begin{tikzpicture}[scale=.6]
  \node[int] (v) at (0,0){};
  \node (v1) at (90:1.3) {$1$};
  \node (v2) at (-30:1.3) {$2$};
  \node (v3) at (-150:1.3) {$3$};
  \draw (v) edge (v1) edge (v2) edge (v3);
  \end{tikzpicture}
  \,,\quad
  \begin{tikzpicture}[scale=.5]
    \node (n1) at (180:2.5) {$1$};
    \node (n2) at (0:2.5) {$2$};
  \node[int] (v1) at (-1,0){};\node[int] (v2) at (0,1){};\node[int] (v3) at (1,0){};\node[int] (v4) at (0,-1){};
  \draw (v1) edge (v2) edge (v4) edge (n1) 
  (v2) edge (v4) 
  (v3) edge (v2) edge (v4) edge (n2);
  \end{tikzpicture}
\]
An orientation of such a graph with numbered hairs is defined as before, but omitting the ordering of hairs, and hence omitting reference to the number $m$.

There generalized or colored hairy graph complex $\CHG_{n}(r)$ is spanned by oriented admissible hairy graphs with $r$ numbered hairs, modulo the same equivalence relations as above.
The differential $d$ is again given by edge contraction as before.

The complex $\CHG_{n}(r)$ carries a natural action of the symmetric group $S_r$ by renumbering hairs, and $\HG_{m,n}$ can be obtained from $\CHG_{n}(r)$ by taking $S_r$-(anti-)invariants and a direct sum over $r$. 

The complexes $\CHG_{n}(r)$ also appear in the literature. In particular Chan-Galatius-Payne \cite{CGP2} have shown that for $n$ even their homology computes the top weight part of the homology of the moduli spaces of curves, up to degree shifts.
Furthermore, the rational homotopy groups of link spaces may be expressed through the homology of $\CHG_{n}(r)$, see 
\cite{SonghafouoTsopmeneTurchin2015}.

\subsubsection{Vanishing result}
There are known vanishing results for the hairy graph homology, outside a certain range. We formulate them here in terms of the number of vertices rather than the homological degree.

\begin{prop}\label{prop:HGC vanishing}
Let $V_{g,h}$ represent the part of loop order $g$ and $h$ hairs of either of the complexes $\HG_{m,n}$ or $\CHG_n(r)$.
Then the nontrivial homology of $V_{g,h}$ can be represented by linear combinations of graphs with their number of internal vertices $v$ in the range 
\[
g+h-2 \leq v \leq 2g+h-2.
\]
\end{prop}
\begin{proof}
The upper bound on $v$ follows by using that the internal vertices must be at least trivalent.
The lower bound is a reformulation of \cite[Corollary 16]{BrunWillwacher22}.
\end{proof}

\subsection{Forested graph complexes}\label{sec:def_forested}
An admissible forested graph with $r$ hairs is the following data:
\begin{itemize}
\item An undirected connected 
 graph with $r$ numbered univalent and an arbitrary number of $\geq 3$-valent vertices.
 The $r$ univalent vertices we call the hairs, and the other vertices internal vertices.
\item In addition, there is a distinguished forest, that is, a subset $T\subset E\gamma$ of the set of edges, that does not contain any edge incident to a hair vertex, and that does not contain loops. We call the set of edges in the forest $T$ the marked edges.
\end{itemize}
Mind that such graphs are a priori allowed to have tadpoles and multiple edges. 
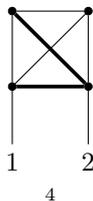
\[
\begin{tikzpicture}
  \node[int] (v1) at (0,0) {};
  \node[int] (v2) at (1,0) {};
  \node[int] (v3) at (1,1) {};
  \node[int] (v4) at (0,1) {};
  \node[] (e1) at (0,-1) {$ 1$};
  \node[] (e2) at (1,-1) {$ 2$};
  \draw (v1) edge (e1) edge[fatedge] (v2) edge (v3) edge (v4) 
  (v2) edge (e2) edge (v3) edge[fatedge] (v4) 
  (v3) edge (v4);
\end{tikzpicture}
\]
Depending on an integer $n\in \{0,1\}$ we define the orientation of an admissible forested graph as the following data:

\begin{itemize}
  \item For $n=0$, an orientation is an ordering of the set of marked edges, up to even permutations.
  \item For $n=1$, an orientation is an ordering of the set of internal vertices, half-edges, and unmarked edges, up to even permutation.
\end{itemize}

We define the graph complex $\FG_n(r)$ to be spanned by triples $(\gamma,T,\ori)$ consisting of a graph $\gamma$  with $r$ hairs, a distinguished forest $T\subset E\gamma$ and an orientation $\ori$, modulo the following equivalence relations:
\begin{itemize}
  \item We identify isomorphic oriented forested graphs. Here isomorphisms are required to preserve the numbering of the hairs, and preserve (set-wise) the distinguished forest.
  \item We identify opposite orientations up to sign, $(\gamma,\ori)=-(\gamma,-\ori)$.
\end{itemize}

We define the homological degree of an admissible forested graph with $m$ internal marked edges to be $m$.
The graded vector spaces $\FG_n(r)$ carry two differentials:
\begin{itemize}
\item The edge contraction differential $d_c$ acts by contracting a marked edge, similar to \eqref{equ:d def ordinary} and \eqref{equ:d def ordinary 2}. For $n$ even we define
\[
d_c(\gamma,T,\ori) = \sum_{e\in T}
(\gamma/e, T\setminus \{e\}, \partial_e \ori ).
\]
For $n$ odd we set 
\[
d_c(\gamma,T,\ori) = \sum_{e\in T}
(\gamma/e, T\setminus \{e\}, v_e\wedge ( \partial_{v_e^1}\partial_{v_e^2}\partial_{h_e^1}\partial_{h_e^2}  \ori)).
\]

\item The unmarking differential $d_u$ acts by unmarking a marked edge. For $n$ even it is defined by 
\[
  d_u(\gamma,T,\ori) =  \sum_{e\in T} (\gamma,T\setminus \{e\},
  \partial_e \ori),
\]
and for $n$ odd,
\begin{equation*}
  d_u(\gamma,T,\ori) =  \sum_{e\in T} (\gamma,T\setminus \{e\},
  e\wedge \ori).
\end{equation*}
\end{itemize}

We leave it as an exercise for the reader to check that the following result, see \cite{Clivio}:
\begin{lemma}
We have $d_c^2 = d_cd_u+d_ud_c=d_u^2=0$.
\end{lemma}

The primary differential on the complex $\FG_n(r)$ we consider is $d=d_c+d_u$. With this differential the graph complex computes the homology of a family of groups $\Gamma_{g,r}$ (see \cite{HatcherVogtmann2004}) that generalize the outer automorphism groups of free groups $\Out(F_g)=\Gamma_{g,0}$ and the automorphism groups of free groups $\Aut(F_g)=\Gamma_{g,1}$. More precisely, one has the following result:

\begin{thm}[after Conant, Kassabov, Hatcher, Vogtmann]
Let $g,r \geq 0$ be such that $2g+r\geq 3$. Then we have that 
\begin{align}
\label{equ:FG0 out}
H_k(\FG_0^{g\text{-loop}}(r),d_c+d_u) &\cong H_k(\Gamma_{g,r};\Q)
\\
\label{equ:FG1 out}
H_k(\FG_1^{g\text{-loop}},d_c+d_u) &\cong H_{k}(\Gamma_{g,r};\sgn),
\end{align}
with $\Q$ the trivial representation of $\Gamma_{g,r}$, and $\sgn$ the sign representation.
\end{thm}
\begin{proof}[Proof sketch]
Although the result is (mostly) shown in the literature, we do not know a reference in which it appears in the stated form.
First, \eqref{equ:FG0 out} for $r=0$ can be found in \cite{ConantVogtmann2003, Kontsevich1993}.
The general case (i.e., general $r$) of \eqref{equ:FG0 out} is shown in \cite[Theorem 11.1]{ConantKassabovVogtmann2013}.

This latter proof can be generalized to other representations of $\Gamma_{g,r}$ to obtain in particular \eqref{equ:FG1 out}. More precisely, as in said proof, one uses again \cite[Proposition 8.7]{ConantKassabovVogtmann2013} with $X=A_{g,r}$ the spaces of \cite{HatcherVogtmann2004}, but with $M$ the sign representation.
This expresses the right-hand side of \eqref{equ:FG1 out} as the homology of $A_{g,r}/\Gamma_{g,r}$, but with twisted coefficients. The local system that gives rise to the twist is, on the cell of $A_{g,r}$ of a graph $\gamma$, the one-dimensional vector spaces $\Det(H_1(\gamma))=\wedge^g H_1(\gamma)$.

Similarly, the left-hand side of \eqref{equ:FG1 out} is identified with the homology of $A_{g,r}/\Gamma_{g,r}$ with coefficients in the local system $O$ such that 
a generator of $O(\gamma)$ is given by an ordering on the set of half-edges, edges and vertices of $\gamma$, modulo even permutations.
Equivalently, 
$
O(\gamma)
=
\Det(C_1(\gamma))\Det(C_0(\gamma))^{-1},
$
where $C_\bullet(\gamma)$ is the simplicial chain complex of $\gamma$.
Finally, we leave it as an exercise that the local systems $O(-)$ and $\Det(H_1(-))$ are isomorphic.
\end{proof}

There are furthermore known stabilization and vanishing results for the homology of $\Out(F_g)$ due to Hatcher-Vogtmann and Galatius that we summarize as follows.

\begin{thm}[\cite{HatcherVogtmann1998,HatcherVogtmann2004,Galatius2011} ]\label{thm:Out vanishing}
The homology $H_k(\Out(F_g);\Q)$ vanishes if $5k< 4g-5$. 
\end{thm}

The graph complexes $\FG_n(r)$ are relatively large.
Fortunately however, they contain slightly smaller quasi-isomorphic subcomplexes, that we shall describe in the following sections.

\subsubsection{Bridgeless version}
A bridge in a connected graph $\gamma$ is an edge $e$ such that the graph $\gamma\setminus e$ obtained by removing $e$ is disconnected.
The graph $\Gamma$ is called bridgeless if none of its edges is a bridge.
Note that by convention we do not count external legs as edges, and hence there can be bridgeless graphs with external legs.
We denote by 
\[
\FG_n^{bl} \subset \FG_n
\]
the subcomplex spanned by bridgeless forested graphs.
\begin{prop}\label{prop:FGC bl1}
The inclusion $\FGC_n^{bl,g\lp} \subset \FGC_n^{g\lp}$ is a quasi-isomorphism for all $n$ and all $g\geq 1$.
\end{prop} 
\begin{proof}
  For $\gamma$ a forested graph we denote by $m(\gamma)$ the number of marked non-bridge edges, and by $b(\Gamma)$ the number of (marked or unmarked) bridges in $\Gamma$.

We endow $\FG_n$ with a filtration 
\[
0=\mF^{-1} \FG_n \subset \mF^0 \FG_n \subset \cdots 
\]
such that $\mF^p \FG_n$ is spanned by graphs $\Gamma$ such that 
\[
  m(\Gamma) + b(\Gamma) \leq p.
\]
Note that the contraction of a marked edge always reduces the number on the left by at least one. 
The unmarking of an edge reduces the number by one if the edge is not a bridge, and leaves the number the same if it is.
Let us denote the summand of $d_u$ that acts on bridge edges only by $d_u'$.
Then the associated graded of $\FG_n$ with respect to the above filtration is identified with $(\FG_n,d_u')$.

We may restrict the filtration on $\FG_n^{bl}$ as well.
It is then sufficient to check that the inclusion of associated graded complexes 
\[
(\FG_n^{bl},0) = \gr \FG_n^{bl} \subset 
(\FG_n,d_u') = \gr \FG_n
\]
is a quasi-isomorphism. 
But define on $(\FG_n,d_u')$ a homotopy $h$ such 
\[
h\Gamma = \sum_{e\in B\Gamma} \Gamma \text{ mark } e,  
\]
with $B\Gamma$ the set of bridges.
Then 
\[
(d_u'h+hd_u') \Gamma = b(\Gamma) \Gamma.
\]
Hence the homology is indeed identified with the bridgeless part $\FG_n^{bl}$.
\end{proof}

\subsubsection{Homology of $d_c$}
A second simplification of the complexes $\FG_n$ is possible due to the homology of the differential $d_c$ being concentrated in ``top'' (in a suitable sense) degree, corresponding to trivalent graphs.
This is known for (a version of) the standard forested graph complex $\FG_0(0)$ due to work of Kontsevich and Conant-Vogtmann. We need to slightly generalize the result.

For a forested graph $\Gamma$ let us define the excess of $\Gamma$ as 
  \[
  exc(\Gamma) = \sum_{v\in V\Gamma} (\text{valence}(v)-3) \geq 0.
  \]
  For example, graphs of excess zero are the trivalent graphs.
  Let $\FGC_n^{bl,e-exc}\subset \FGC_n^{bl}$ be the graded subspace spanned by graphs of excess $e$. 

  \begin{prop}
    \label{prop:FGC dc bottom}
    \begin{itemize}
      \item 
    For each $n$ and $g\geq 0$ the homology of $\FG_{n}^{g\lp}$ with respect to the edge contraction differential $d_c$ is concentrated in excess zero, corresponding to graphs all of whose vertices are trivalent.
    \item 
    For each $n$ and $g\geq 1$ the homology of $\FG_{n}^{bl,g\lp}$ with respect to the edge contraction differential $d_c$ is concentrated in excess zero, corresponding to graphs all of whose vertices are trivalent.
    \end{itemize}
  \end{prop}
  The first statement follows from the Koszulness of the cyclic commutative operad and was already used in \cite{Ohashi, Bartholdi}.
  We shall give a proof of the second statement in Appendix \ref{app:FGC dc bottom proof} below.

  \begin{cor}\label{cor:Kn}
    Let $K_n\subset \FGC_n^{bl,0\exc}$ be the kernel of the contraction differential 
    \[
    d_c : \FGC_n^{bl,0\exc}  \to \FGC_n^{bl,1\exc}. 
    \]
    Then the differential $d_u$ restricts to $K_n$ and for every $g\geq 1$ we have that the inclusion 
    \[
      (K_n^{g\lp}, d_u) \subset (\FGC_n^{bl,g\lp}, d_c+d_u)
    \]
    is a quasi-isomorphism.
  \end{cor}

\subsection{Graph complexes in general, from a computational viewpoint}
Generally, a graph complex is a differential graded vector space with a basis indexed by isomorphism classes of some sort of combinatorial graphs, with a combinatorially defined differential.
From a computational standpoint, the computation of graph homology has the following characteristic features.
\begin{enumerate}
\item The computation of a basis of the graph complex is a graph enumeration problem.
\item Both the computation of the basis and of the differential involve keeping track of graph isomorphisms.
\item The matrix of the differential is very sparse.
 For example, the matrix of the edge contraction differential has at most as many non-zero entries per column as the maximum number of edges in the graphs considered. 
\end{enumerate}
We also note that for numeric computations we have to restrict to graph complexes that are degree-wise finite dimensional.
For all the complexes described above, this is true if one restricts to the part of fixed loop order.
This is possible since the differentials do not alter the loop order, and hence the graph complexes split into a direct sum of subcomplexes.

\section{The \GH library}\label{sec:GH}

Implementing a graph complex on the computer can be a laborious and error-prone task. Much of the effort is however associated to relatively uninteresting and tedious problems, like organizing files and managing long running computations, that have nothing to do with the graph complex at hand. We developed the \GH framework and library to take over these chores. The goal is that the developer can focus on the implementation of the mathematical core content of the graph complex.

\GH is written in python and makes essential use of \textsc{Sage} \cite{sage}, \textsc{Nauty} \cite{nauty} and \textsc{LinBox} \cite{LinBox}.
It contains the following main components.

\subsection{Basis generation and graph isomorphisms}
All graphs used in \GH are undirected simple graphs with colored vertices. They are internally represented in the \textsc{Sage} \texttt{Graph} format for undirected graphs.
If the graph complex the user wants to implement  contains graphs with more structure, as is usually the case, they need to be encoded as colored simple graphs.

\begin{ex}
For example, the forested graphs above with $r$ distinguishable hairs can be encoded as colored simple graphs with vertices of $r+2$ colors as follows.
Color 1 is for the internal vertices, with the marked edges encoded as edges between these vertices. A non-marked edge is encoded by a bivalent vertex of color 2 and its two adjacent edges. Each external leg is represented by a univalent vertex of a unique color $3,\dots,r+2$.
  
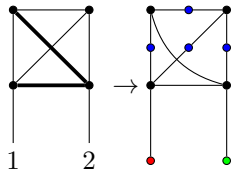
\[
\begin{tikzpicture}
  \node[int] (v1) at (0,0) {};
  \node[int] (v2) at (1,0) {};
  \node[int] (v3) at (1,1) {};
  \node[int] (v4) at (0,1) {};
  \node[] (e1) at (0,-1) {$ 1$};
  \node[] (e2) at (1,-1) {$ 2$};
  \draw (v1) edge (e1) edge[fatedge] (v2) edge (v3) edge (v4) 
  (v2) edge (e2) edge (v3) edge[fatedge] (v4) 
  (v3) edge (v4);
\end{tikzpicture}
\to
\begin{tikzpicture}
  \node[int] (v1) at (0,0) {};
  \node[int] (v2) at (1,0) {};
  \node[int] (v3) at (1,1) {};
  \node[int] (v4) at (0,1) {};
  \node[int,fill=red] (e1) at (0,-1) {};
  \node[int,fill=green] (e2) at (1,-1) {};
  \node[int, fill=blue] (w1) at (0,0.5) {};
  \node[int, fill=blue] (w2) at (0.5,0.5) {};
  \node[int, fill=blue] (w3) at (0.5,1) {};
  \node[int, fill=blue] (w4) at (1,0.5) {};
  \draw (v1) edge (e1) edge (v2) 
  (w2) edge (v1) edge (v3) 
   (w1) edge (v4) edge (v1) 
  (v2) edge (e2)  edge[bend left] (v4) 
  (w3) edge (v3) edge (v4)
  (w4) edge (v3)edge (v2);
\end{tikzpicture}
\]
\end{ex}

The graph vector space is decomposed into finite dimensional subspaces, characterised by a specific combination of parameters like a fixed number of vertices and loops and is represented by the abstract class \texttt{GraphVectorSpace}.
The user of \GH implements this class and mainly provides two functions:
First, a method \texttt{perm\_sign()}, which determines the sign by which an isomorphism acts on a graph. And second, a graph generation routine, overriding the method \texttt{get\_generating\_graphs()}, which is supposed to return a generating set of graphs whose isomorphism classes span the graph vector space, not necessarily freely.
The user can use the \textsc{Nauty} graph library \cite{nauty} to list simple or bipartite graphs. A corresponding interface is provided. From the generating set of graphs the method \texttt{build\_basis()} builds a basis consisting of canonically labelled graphs by dropping all graphs with odd automorphisms. Finally the basis is stored as a list of \texttt{graph6} strings in a corresponding basis file. 

\subsection{Linear operators} 
Linear operators acting on graph vector spaces like contracting edges or splitting edges in one or two hairs are represented by the class \texttt{GraphOperator}. Essentially, the user only needs to implement the abstract method \texttt{operate\_on()} that receives a graph of the domain \texttt{GraphVectorSpace} and is supposed to return a linear combination of graphs in the target \texttt{GraphVectorSpace}. The \texttt{GraphOperator} class provides the method \texttt{build\_matrix()} to build the transformation matrix of the operator with respect to the basis of the domain and target graph vector space and stores it in a corresponding matrix file using the \textsc{SMS} format. Furthermore, the method \texttt{compute\_rank()} offers different ways to compute the rank of the operator matrix, which is needed to determine the dimension of the homology. For small matrices \textsc{Sage}'s own method to determine a matrix rank can be used, whereas for large matrices it is adequate to use either the \textsc{LinBox} \cite{LinBox} or the \textsc{Rheinfall} \cite{Rheinfall} library for exact rank computations. 

The class \texttt{Differential} inherits from represents linear operators which are supposed to be a differential. Hence it provides the method \texttt{square\_zero\_test()} which verifies whether the linear operator squares to zero, i.e. is a differential. Furthermore, there are methods for computing and plotting the dimension of the homology of the complex associated with the differential. 

\subsection{Graph complexes}
A graph complex is represented by the class \texttt{GraphComplex}. Fields refer to the underlying vector space as well as a list of linear operators. The method \texttt{test\_anti\_commutativity()} tests whether two linear operators anti-commute and thus whether two differentials build a bicomplex.

\subsection{Bicomplexes}
We also consider the complexes associated with the total differential of bicomplexes. In those cases the vector space is composed of a direct sum of degree slices represented by the class \texttt{DegSlice}, which are themselves direct sums of graph vector spaces representing the subspaces of constant total degree. The total differential of the bicomplex is represented by a collection of operator matrices in the form of the class \texttt{BiOperatorMatrix} whose domain and target vector spaces are degree slices. 
For example, the forested graph complex is encoded as a bicomplex in this manner. 

\subsection{Testing}
Beside the built-in generic tests to verify whether a differential squares to zero and whether two differentials anti-commute, there is the opportunity to compare the generated basis as well as operator matrices with reference data, which is implemented in the classes \texttt{RefGraphVectorSpace} and \texttt{RefOperatorMatrix} respectively.

\section{Experiments and numerical results}
\label{sec:experiments}

\subsection{Foreword on data presentation}
This section contains tables of numerical results of graph homology computations.
In these tables we generally use the following conventions:
\begin{itemize}
\item The rows and columns correspond to the relevant gradings of the graph complex, mostly loop order $l$ and number of vertices $v$. The entry in the $(l,v)$ cell of the table is the computed dimension of the $(l,v)$-part of the homology of the graph complex. Entries "?" mean that we could not compute the homology dimension at that place, mostly due to increasing computational demands.
\item If the graph complex has a natural action of the symmetric group, then we have computed the decomposition of the homology into irreducible representations of $S_n$.
For example, a table entry 
\[
3\, (s_{[3]}+s_{[2,1]})  
\]
means that the total dimension of the respective homology group is 3, and this 3-dimensional space decomposes into a trivial representation of $S_3$ and the 2-dimensional irreducible representation, corresponding to the partition $2+1$. In particular, the above notation does \emph{not} mean that both of these ireducibles occur three times each. 
\item Entries "-" mean that the respective graded piece of the graph complex is trivial, i.e., it contains no graphs and hence the associated homology is trivial. Entries "0" mean that the graph complex is nontrivial, but the homology turned out zero nevertheless. The reader may treat "$0$" and "-" the same, but the distinction indicates where the graph complex is an efficient or not so efficient representation of the underlying homological problem.
\item Where applicable, we have indicated known theoretical vanishing results by shading cells that are known to contain no homology.
\item Some tables also contain reference Euler characteristic data from the literature. This is included both as confirmation of our results, and an indicator of the expected amount of homology in the respective row of the table.
\item Note that our computations were mostly conducted over finite fields $\mathbb F_p$ rather than $\mathbb Q$, see the methods sections below for a discussion of the obtained results.
\end{itemize}

More data and the $\GH$ source code can be found on the project's github page.

\begin{center}
\url{https://github.com/sibrun/GH}
\end{center}

\subsection{Ordinary graphs}

Considering the ordinary graph complex $G_n$ with the contract-edges-differential we found the dimensions of the homology shown in Figure \ref{fig:ordinary}.

We note that a subset of the data has been computed earlier by Bar-Natan and McKay \cite{BarNatanMcKay}.
Furthermore, a subset of the tables has appeared in \cite{KWZ1}, based on previous computations of the second author, mostly using floating point arithmetic.

\begin{figure}[H]
    \centering

\smallskip
odd $n$ 

\begin{tabular}{|g|M|M|M|M|M|M|M|M|M|M|M|M|M|M|M|M|M|M|M|M|}
 \rowcolor{Gray}
\hline
l,v
 & 4 & 5 & 6 & 7 & 8 & 9 & 10 & 11 & 12 & 13 & 14 & 15 & 16 & 17 & 18 & 19 & 20 & 21 & $\chi$ & $\chi_{ref}$\\ 
\hline
3 & 1 & - \cellcolor{LightGray} & - \cellcolor{LightGray} & - \cellcolor{LightGray} & - \cellcolor{LightGray} & - \cellcolor{LightGray} & - \cellcolor{LightGray} & - \cellcolor{LightGray} & - \cellcolor{LightGray} & - \cellcolor{LightGray} & - \cellcolor{LightGray} & - \cellcolor{LightGray} & - \cellcolor{LightGray} & - \cellcolor{LightGray} & - \cellcolor{LightGray} & - \cellcolor{LightGray} & - \cellcolor{LightGray} & - \cellcolor{LightGray} & 1 & 1\\ 
4 & - \cellcolor{LightGray} & 0 & 1 & - \cellcolor{LightGray} & - \cellcolor{LightGray} & - \cellcolor{LightGray} & - \cellcolor{LightGray} & - \cellcolor{LightGray} & - \cellcolor{LightGray} & - \cellcolor{LightGray} & - \cellcolor{LightGray} & - \cellcolor{LightGray} & - \cellcolor{LightGray} & - \cellcolor{LightGray} & - \cellcolor{LightGray} & - \cellcolor{LightGray} & - \cellcolor{LightGray} & - \cellcolor{LightGray} & 1 & 1\\ 
5 & - \cellcolor{LightGray} & 0 \cellcolor{LightGray} & 0 & 0 & 2 & - \cellcolor{LightGray} & - \cellcolor{LightGray} & - \cellcolor{LightGray} & - \cellcolor{LightGray} & - \cellcolor{LightGray} & - \cellcolor{LightGray} & - \cellcolor{LightGray} & - \cellcolor{LightGray} & - \cellcolor{LightGray} & - \cellcolor{LightGray} & - \cellcolor{LightGray} & - \cellcolor{LightGray} & - \cellcolor{LightGray} & 2 & 2\\ 
6 & - \cellcolor{LightGray} & 0 \cellcolor{LightGray} & 0 \cellcolor{LightGray} & 1 & 0 & 0 & 2 & - \cellcolor{LightGray} & - \cellcolor{LightGray} & - \cellcolor{LightGray} & - \cellcolor{LightGray} & - \cellcolor{LightGray} & - \cellcolor{LightGray} & - \cellcolor{LightGray} & - \cellcolor{LightGray} & - \cellcolor{LightGray} & - \cellcolor{LightGray} & - \cellcolor{LightGray} & 1 & 1\\ 
7 & - \cellcolor{LightGray} & - \cellcolor{LightGray} & 0 \cellcolor{LightGray} & 0 \cellcolor{LightGray} & 0 & 1 & 0 & 0 & 3 & - \cellcolor{LightGray} & - \cellcolor{LightGray} & - \cellcolor{LightGray} & - \cellcolor{LightGray} & - \cellcolor{LightGray} & - \cellcolor{LightGray} & - \cellcolor{LightGray} & - \cellcolor{LightGray} & - \cellcolor{LightGray} & 2 & 2\\ 
8 & - \cellcolor{LightGray} & - \cellcolor{LightGray} & 0 \cellcolor{LightGray} & 0 \cellcolor{LightGray} & 0 \cellcolor{LightGray} & 0 & 0 & 2 & 0 & 0 & 4 & - \cellcolor{LightGray} & - \cellcolor{LightGray} & - \cellcolor{LightGray} & - \cellcolor{LightGray} & - \cellcolor{LightGray} & - \cellcolor{LightGray} & - \cellcolor{LightGray} & 2 & 2\\ 
9 & - \cellcolor{LightGray} & - \cellcolor{LightGray} & 0 \cellcolor{LightGray} & 0 \cellcolor{LightGray} & 0 \cellcolor{LightGray} & 0 \cellcolor{LightGray} & 0 & 0 & 0 & 3 & 0 & 0 & 5 & - \cellcolor{LightGray} & - \cellcolor{LightGray} & - \cellcolor{LightGray} & - \cellcolor{LightGray} & - \cellcolor{LightGray} & 2 & 2\\ 
10 & - \cellcolor{LightGray} & - \cellcolor{LightGray} & 0 \cellcolor{LightGray} & 0 \cellcolor{LightGray} & 0 \cellcolor{LightGray} & 0 \cellcolor{LightGray} & 0 \cellcolor{LightGray} & 0 & 0  & 0  & 0  & 5  & 0  & 0 & 6 & - \cellcolor{LightGray} & - \cellcolor{LightGray} & - \cellcolor{LightGray} & 1 & 1\\ 
11 & - \cellcolor{LightGray} & - \cellcolor{LightGray} & - \cellcolor{LightGray} & 0 \cellcolor{LightGray} & 0 \cellcolor{LightGray} & 0 \cellcolor{LightGray} & 0 \cellcolor{LightGray} & ? \cellcolor{LightGray} & ? & ? & ? & ? & ? & ? & ? & ? & 8  & - \cellcolor{LightGray} & ? & 3\\ 
\hline
\end{tabular}

\smallskip
even $n$

\begin{tabular}{|g|M|M|M|M|M|M|M|M|M|M|M|M|M|M|M|M|M|M|M|M|}
 \rowcolor{Gray}
\hline
l,v
 & 4 & 5 & 6 & 7 & 8 & 9 & 10 & 11 & 12 & 13 & 14 & 15 & 16 & 17 & 18 & 19 & 20 & 21 & $\chi$ & $\chi_{ref}$\\ 
\hline
3 & 1 & - \cellcolor{LightGray} & - \cellcolor{LightGray} & - \cellcolor{LightGray} & - \cellcolor{LightGray} & - \cellcolor{LightGray} & - \cellcolor{LightGray} & - \cellcolor{LightGray} & - \cellcolor{LightGray} & - \cellcolor{LightGray} & - \cellcolor{LightGray} & - \cellcolor{LightGray} & - \cellcolor{LightGray} & - \cellcolor{LightGray} & - \cellcolor{LightGray} & - \cellcolor{LightGray} & - \cellcolor{LightGray} & - \cellcolor{LightGray} & 1 & 1\\ 
4 & - \cellcolor{LightGray} & 0 & 0 & - \cellcolor{LightGray} & - \cellcolor{LightGray} & - \cellcolor{LightGray} & - \cellcolor{LightGray} & - \cellcolor{LightGray} & - \cellcolor{LightGray} & - \cellcolor{LightGray} & - \cellcolor{LightGray} & - \cellcolor{LightGray} & - \cellcolor{LightGray} & - \cellcolor{LightGray} & - \cellcolor{LightGray} & - \cellcolor{LightGray} & - \cellcolor{LightGray} & - \cellcolor{LightGray} & 0 & 0\\ 
5 & - \cellcolor{LightGray} & 0 \cellcolor{LightGray} & 1 & 0 & 0 & - \cellcolor{LightGray} & - \cellcolor{LightGray} & - \cellcolor{LightGray} & - \cellcolor{LightGray} & - \cellcolor{LightGray} & - \cellcolor{LightGray} & - \cellcolor{LightGray} & - \cellcolor{LightGray} & - \cellcolor{LightGray} & - \cellcolor{LightGray} & - \cellcolor{LightGray} & - \cellcolor{LightGray} & - \cellcolor{LightGray} & 1 & 1\\ 
6 & - \cellcolor{LightGray} & 0 \cellcolor{LightGray} & 0 \cellcolor{LightGray} & 0 & 0 & 0 & 1 & - \cellcolor{LightGray} & - \cellcolor{LightGray} & - \cellcolor{LightGray} & - \cellcolor{LightGray} & - \cellcolor{LightGray} & - \cellcolor{LightGray} & - \cellcolor{LightGray} & - \cellcolor{LightGray} & - \cellcolor{LightGray} & - \cellcolor{LightGray} & - \cellcolor{LightGray} & 1 & 1\\ 
7 & - \cellcolor{LightGray} & - \cellcolor{LightGray} & 0 \cellcolor{LightGray} & 0 \cellcolor{LightGray} & 1 & 0 & 0 & 0 & 0 & - \cellcolor{LightGray} & - \cellcolor{LightGray} & - \cellcolor{LightGray} & - \cellcolor{LightGray} & - \cellcolor{LightGray} & - \cellcolor{LightGray} & - \cellcolor{LightGray} & - \cellcolor{LightGray} & - \cellcolor{LightGray} & 1 & 1\\ 
8 & - \cellcolor{LightGray} & - \cellcolor{LightGray} & 0 \cellcolor{LightGray} & 0 \cellcolor{LightGray} & 0 \cellcolor{LightGray} & 1 & 0 & 0 & 1 & 0 & 0 & - \cellcolor{LightGray} & - \cellcolor{LightGray} & - \cellcolor{LightGray} & - \cellcolor{LightGray} & - \cellcolor{LightGray} & - \cellcolor{LightGray} & - \cellcolor{LightGray} & 0 & 0\\ 
9 & - \cellcolor{LightGray} & - \cellcolor{LightGray} & 0 \cellcolor{LightGray} & 0 \cellcolor{LightGray} & 0 \cellcolor{LightGray} & 0 \cellcolor{LightGray} & 1 & 0 & 0 & 1 & 0 & 0 & 0 & - \cellcolor{LightGray} & - \cellcolor{LightGray} & - \cellcolor{LightGray} & - \cellcolor{LightGray} & - \cellcolor{LightGray} & 0 & 0\\ 
10 & - \cellcolor{LightGray} & - \cellcolor{LightGray} & 0 \cellcolor{LightGray} & 0 \cellcolor{LightGray} & 0 \cellcolor{LightGray} & 0 \cellcolor{LightGray} & 0 \cellcolor{LightGray} & 1 & 0  & 0  & 2  & 0  & 0  & 0 & 1 & - \cellcolor{LightGray} & - \cellcolor{LightGray} & - \cellcolor{LightGray} & 2 & 2\\ 
11 & - \cellcolor{LightGray} & - \cellcolor{LightGray} & - \cellcolor{LightGray} & 0 \cellcolor{LightGray} & 0 \cellcolor{LightGray} & 0 \cellcolor{LightGray} & 0 \cellcolor{LightGray} & ? \cellcolor{LightGray} & ? & ? & ? & ? & ? & ? & ? & ? & 0  & - \cellcolor{LightGray} & ? & 1\\ 
\hline
\end{tabular}

    \caption{Homology dimensions for the ordinary graph complex $H(\G_n)$. The rows correspond to loop order $l$ and the columns to the number of vertices $v$.
    The shaded cells must be zero by known vanishing theorems. Concretely, the upper bound is due to the trivalence condition on vertices of graphs in the graph complex. The lower bound is due to \cite{Willwacher2014}. We remark that only the case of even $n$ is stated in \cite{Willwacher2014}. However, for odd $n$ the corresponding result can be shown analogously. The reference Euler characteristics have been taken from \cite{WillwacherZivkovic15}. }
    \label{fig:ordinary}
\end{figure}

\subsubsection{Methods}\label{sec:ordinary methods}

We consider the subspaces 
\begin{align*}
V_{g,n,v} \subset \G_n
\end{align*}
spanned by graphs of loop order $g$ with $v$ vertices, and the linear operators
\begin{align*}
  d_{v} := d\mid_{V_{g,n,v}}  : V_{g,n,v} \to V_{g,n,v-1}.
\end{align*}
The dimension of the $g$-loop, $v$-vertex part of the homology is then 
\begin{equation}\label{equ:Dv ord rank}
  D_v := \dim V_{g,n,v} - \rank d_{v} -\rank d_{v+1}.
\end{equation}

We generate a basis of $V_{g,n,v}$ as follows. We first produce a list of all isomorphism classes of connected 1-vertex irreducible $g$-loop graphs with $v$ vertices of valence $\geq 3$.
To this end we use B. McKay's \textsc{Nauty} library \cite{nauty}. We canonically label the graphs and remove all that have odd symmetries, and are hence zero in $\G_n$. This yields a basis of 
$V_{g,n,v}$, and hence its dimension.

We next produce the matrices of the differentials $d_{v}$ in a custom \textsc{Sage} program. These matrices are very sparse, with integer entries.
Finally, we use \textsc{LinBox} \cite{LinBox} to compute the ranks of these matrices, and hence the numbers $D_v$ above.

To avoid memory overflow a part of the rank computations is conducted over a prime (32189). 
Note that the rank modulo a prime of an integer matrix is always less than or equal to the actual (rational) rank of the matrix.
Furthermore, the rank computation algorithm used internally by LinBox provides an approximation to the actual rank, that is only guaranteed to be a lower bound to the actual rank. (That said, it is correct with high probability.)
Overall, this means that we only compute a lower bound (and approximation)
\begin{equation}\label{equ:rv approx}
  r_v \leq \rank d_{v}
\end{equation}
to the desired ranks of $d_v$.
From \eqref{equ:Dv ord rank} we then see that 
\[
  D_v \leq \dim V_{g,n,v} -r_v -r_{v+1}.
\]
Suppose that $v$ is such that the right-hand side is $0$. Then, since $D_v\geq 0$ we deduce $D_v=0$. But from \eqref{equ:Dv ord rank} and \eqref{equ:rv approx} we then obtain that
\begin{align*}
  r_v &= \rank d_{v}
  &
  r_{v+1} &= \rank d_{v+1},
\end{align*}
that is, our approximations to the ranks are exact.
This argument has been used before by Bartholdi \cite{Bartholdi}.

Looking at the tables in Figure \ref{fig:ordinary} we see that no two horizontally consecutive cells are non-zero, except for the (8,20)-entry of the first table. Hence we can deduce that our computation of $D_v$ is actually correct over $\Q$ and rigorous, except for the entry "8" in the even-edge graph complex, with 11 loops and 20 vertices.

\subsection{Testing Merkulov's conjecture}

We checked Merkulov's conjecture, Conjecture \ref{conj:merkulov}, numerically, up to loop order 10.
The results are displayed in Figure \ref{fig:merkulov}.
As one can see comparing to Figure \ref{fig:ordinary}, Merkulov's conjecture appears to be true up to loop order 10.
We shall note, however, that we partially used finite-field arithmetic, see the discussion below.

\begin{figure}[H]
  \centering

\smallskip
odd $n$ 

\begin{tabular}{|g|M|M|M|M|M|M|M|M|M|M|M|M|M|M|M|M|M|M|}
 \rowcolor{Gray}
\hline
l,v
& 4 & 5 & 6 & 7 & 8 & 9 & 10 & 11 & 12 & 13 & 14 & 15 & 16 & 17 & 18 & 19 & 20 & 21\\ 
\hline
3 & 1 & - \cellcolor{LightGray} & - \cellcolor{LightGray} & - \cellcolor{LightGray} & - \cellcolor{LightGray} & - \cellcolor{LightGray} & - \cellcolor{LightGray} & - \cellcolor{LightGray} & - \cellcolor{LightGray} & - \cellcolor{LightGray} & - \cellcolor{LightGray} & - \cellcolor{LightGray} & - \cellcolor{LightGray} & - \cellcolor{LightGray} & - \cellcolor{LightGray} & - \cellcolor{LightGray} & - \cellcolor{LightGray} & - \cellcolor{LightGray}\\ 
4 & 0 \cellcolor{LightGray} & 0 & 1 & - \cellcolor{LightGray} & - \cellcolor{LightGray} & - \cellcolor{LightGray} & - \cellcolor{LightGray} & - \cellcolor{LightGray} & - \cellcolor{LightGray} & - \cellcolor{LightGray} & - \cellcolor{LightGray} & - \cellcolor{LightGray} & - \cellcolor{LightGray} & - \cellcolor{LightGray} & - \cellcolor{LightGray} & - \cellcolor{LightGray} & - \cellcolor{LightGray} & - \cellcolor{LightGray}\\ 
5 & 0 \cellcolor{LightGray} & 0 \cellcolor{LightGray} & 0 & 0 & 2 & - \cellcolor{LightGray} & - \cellcolor{LightGray} & - \cellcolor{LightGray} & - \cellcolor{LightGray} & - \cellcolor{LightGray} & - \cellcolor{LightGray} & - \cellcolor{LightGray} & - \cellcolor{LightGray} & - \cellcolor{LightGray} & - \cellcolor{LightGray} & - \cellcolor{LightGray} & - \cellcolor{LightGray} & - \cellcolor{LightGray}\\ 
6 & - \cellcolor{LightGray} & 0 \cellcolor{LightGray} & 0 \cellcolor{LightGray} & 1 & 0 & 0 & 2 & - \cellcolor{LightGray} & - \cellcolor{LightGray} & - \cellcolor{LightGray} & - \cellcolor{LightGray} & - \cellcolor{LightGray} & - \cellcolor{LightGray} & - \cellcolor{LightGray} & - \cellcolor{LightGray} & - \cellcolor{LightGray} & - \cellcolor{LightGray} & - \cellcolor{LightGray}\\ 
7 & - \cellcolor{LightGray} & - \cellcolor{LightGray} & 0 \cellcolor{LightGray} & 0 \cellcolor{LightGray} & 0 & 1 & 0 & 0 & 3 & - \cellcolor{LightGray} & - \cellcolor{LightGray} & - \cellcolor{LightGray} & - \cellcolor{LightGray} & - \cellcolor{LightGray} & - \cellcolor{LightGray} & - \cellcolor{LightGray} & - \cellcolor{LightGray} & - \cellcolor{LightGray}\\ 
8 & - \cellcolor{LightGray} & - \cellcolor{LightGray} & - \cellcolor{LightGray} & 0 \cellcolor{LightGray} & 0 \cellcolor{LightGray} & 0 & 0 & 2 & 0 & 0 & 4 & - \cellcolor{LightGray} & - \cellcolor{LightGray} & - \cellcolor{LightGray} & - \cellcolor{LightGray} & - \cellcolor{LightGray} & - \cellcolor{LightGray} & - \cellcolor{LightGray}\\ 
9 & - \cellcolor{LightGray} & - \cellcolor{LightGray} & - \cellcolor{LightGray} & - \cellcolor{LightGray} & 0 \cellcolor{LightGray} & 0 \cellcolor{LightGray} & 0 & 0 & 0 & 3 & 0 & 0 & 5 & - \cellcolor{LightGray} & - \cellcolor{LightGray} & - \cellcolor{LightGray} & - \cellcolor{LightGray} & - \cellcolor{LightGray}\\ 
10 & - \cellcolor{LightGray} & - \cellcolor{LightGray} & - \cellcolor{LightGray} & - \cellcolor{LightGray} & - \cellcolor{LightGray} & 0 \cellcolor{LightGray} & 0 \cellcolor{LightGray} & 0 & 0  & 0  & 0  & 5  & 0  & 0 & 6 & - \cellcolor{LightGray} & - \cellcolor{LightGray} & - \cellcolor{LightGray}\\ 
11 & - \cellcolor{LightGray} & - \cellcolor{LightGray} & - \cellcolor{LightGray} & - \cellcolor{LightGray} & - \cellcolor{LightGray} & - \cellcolor{LightGray} & 0  \cellcolor{LightGray} & 0  \cellcolor{LightGray} & 0  & ? & ? & ? & ? & ? & ? & ? & ? & - \cellcolor{LightGray}\\ 
\hline
\end{tabular}

\smallskip
even $n$

\begin{tabular}{|g|M|M|M|M|M|M|M|M|M|M|M|M|M|M|M|M|M|M|}
 \rowcolor{Gray}
\hline
l,v
& 4 & 5 & 6 & 7 & 8 & 9 & 10 & 11 & 12 & 13 & 14 & 15 & 16 & 17 & 18 & 19 & 20 & 21\\ 
\hline
3 & 1 & - \cellcolor{LightGray} & - \cellcolor{LightGray} & - \cellcolor{LightGray} & - \cellcolor{LightGray} & - \cellcolor{LightGray} & - \cellcolor{LightGray} & - \cellcolor{LightGray} & - \cellcolor{LightGray} & - \cellcolor{LightGray} & - \cellcolor{LightGray} & - \cellcolor{LightGray} & - \cellcolor{LightGray} & - \cellcolor{LightGray} & - \cellcolor{LightGray} & - \cellcolor{LightGray} & - \cellcolor{LightGray} & - \cellcolor{LightGray}\\ 
4 & 0 \cellcolor{LightGray} & 0 & 0 & - \cellcolor{LightGray} & - \cellcolor{LightGray} & - \cellcolor{LightGray} & - \cellcolor{LightGray} & - \cellcolor{LightGray} & - \cellcolor{LightGray} & - \cellcolor{LightGray} & - \cellcolor{LightGray} & - \cellcolor{LightGray} & - \cellcolor{LightGray} & - \cellcolor{LightGray} & - \cellcolor{LightGray} & - \cellcolor{LightGray} & - \cellcolor{LightGray} & - \cellcolor{LightGray}\\ 
5 & 0 \cellcolor{LightGray} & 0 \cellcolor{LightGray} & 1 & 0 & 0 & - \cellcolor{LightGray} & - \cellcolor{LightGray} & - \cellcolor{LightGray} & - \cellcolor{LightGray} & - \cellcolor{LightGray} & - \cellcolor{LightGray} & - \cellcolor{LightGray} & - \cellcolor{LightGray} & - \cellcolor{LightGray} & - \cellcolor{LightGray} & - \cellcolor{LightGray} & - \cellcolor{LightGray} & - \cellcolor{LightGray}\\ 
6 & - \cellcolor{LightGray} & 0 \cellcolor{LightGray} & 0 \cellcolor{LightGray} & 0 & 0 & 0 & 1 & - \cellcolor{LightGray} & - \cellcolor{LightGray} & - \cellcolor{LightGray} & - \cellcolor{LightGray} & - \cellcolor{LightGray} & - \cellcolor{LightGray} & - \cellcolor{LightGray} & - \cellcolor{LightGray} & - \cellcolor{LightGray} & - \cellcolor{LightGray} & - \cellcolor{LightGray}\\ 
7 & - \cellcolor{LightGray} & - \cellcolor{LightGray} & 0 \cellcolor{LightGray} & 0 \cellcolor{LightGray} & 1 & 0 & 0 & 0 & 0 & - \cellcolor{LightGray} & - \cellcolor{LightGray} & - \cellcolor{LightGray} & - \cellcolor{LightGray} & - \cellcolor{LightGray} & - \cellcolor{LightGray} & - \cellcolor{LightGray} & - \cellcolor{LightGray} & - \cellcolor{LightGray}\\ 
8 & - \cellcolor{LightGray} & - \cellcolor{LightGray} & - \cellcolor{LightGray} & 0 \cellcolor{LightGray} & 0 \cellcolor{LightGray} & 1 & 0 & 0 & 1 & 0 & 0 & - \cellcolor{LightGray} & - \cellcolor{LightGray} & - \cellcolor{LightGray} & - \cellcolor{LightGray} & - \cellcolor{LightGray} & - \cellcolor{LightGray} & - \cellcolor{LightGray}\\ 
9 & - \cellcolor{LightGray} & - \cellcolor{LightGray} & - \cellcolor{LightGray} & - \cellcolor{LightGray} & 0 \cellcolor{LightGray} & 0 \cellcolor{LightGray} & 1 & 0 & 0 & 1 & 0 & 0 & 0 & - \cellcolor{LightGray} & - \cellcolor{LightGray} & - \cellcolor{LightGray} & - \cellcolor{LightGray} & - \cellcolor{LightGray}\\ 
10 & - \cellcolor{LightGray} & - \cellcolor{LightGray} & - \cellcolor{LightGray} & - \cellcolor{LightGray} & - \cellcolor{LightGray} & 0 \cellcolor{LightGray} & 0 \cellcolor{LightGray} & 1 & 0  & 0  & 2  & 0  & 0  & 0 & 1 & - \cellcolor{LightGray} & - \cellcolor{LightGray} & - \cellcolor{LightGray}\\ 
11 & - \cellcolor{LightGray} & - \cellcolor{LightGray} & - \cellcolor{LightGray} & - \cellcolor{LightGray} & - \cellcolor{LightGray} & - \cellcolor{LightGray} & 0  \cellcolor{LightGray} & 0  \cellcolor{LightGray} & 2  & ? & ? & ? & ? & ? & ? & ? & ? & - \cellcolor{LightGray}\\ 
\hline
\end{tabular}

  \caption{\label{fig:merkulov} Results of homology computations for Merkulov's graph complex $\GMe_n$.}
\end{figure}

\subsubsection{Methods}\label{sec:ordinary me methods}

For checking Merkulov's conjecture in loop order $g$ there are two relevant subspaces 
\begin{align*}
V_{g,n,v}^{3,4} &\subset \G_n
&
V_{g,n,v}^{5,6} &\subset \G_n,
\end{align*}
with $V_{g,n,v}^{3,4}$ spanned by the $g$-loop graphs with $v$ vertices all of which are 3- or 4-valent, and $V_{g,n,v}^{5,6}$ spanned by graphs with exactly one vertex of valence 5 or 6 and all others 3- or 4-valent. 

Furthermore, we are interested in the parts of the differential $d$,
\begin{align*}
  d_{v,1} : V_{g,n,v}^{3,4} &\to V_{g,n,v-1}^{3,4}
  \\
  d_{v,2} : V_{g,n,v}^{3,4} &\to V_{g,n,v-1}^{5,6}
  \\
  d_{v,12}=\begin{pmatrix}
    d_{v,1} \\ d_{v,2}
    \end{pmatrix} : V_{g,n,v}^{3,4} &\to 
    V_{g,n,v-1}^{3,4} \oplus V_{g,n,v-1}^{5,6}.
\end{align*}

The dimension of the $g$-loop, $v$-vertex part of the homology of $\GMe_n$ is then, by definition,
\[
D_v := \dim \ker d_{v,2}  - \rank d_{v,1}\mid_{\ker d_{v,2}} - \rank d_{v+1,1}\mid_{\ker d_{v+1,2}}.
\]
We can reduce the computation of $D_v$ to the computations of ranks of the operators $d_{v,-}$ and of dimensions of vector spaces as follows:
First note that 
\[
  \dim \ker d_{v,2} = \dim V_{g,n,v}^{3,4} -  \rank d_{v,2}.
\]
Next, 
\[
  \ker d_{v,1}\mid_{\ker d_{v,2}}
  = 
  \ker d_{v,12},
\]
and hence 
\begin{align*}
  \rank d_{v,1}\mid_{\ker d_{v,2}} 
  &=
  \dim \ker d_{v,2} 
  -
  \dim \ker d_{v,1}\mid_{\ker d_{v,2}}
  \\&=
  \dim V_{g,n,v}^{3,4} -  \rank d_{v,2}
  - \dim V_{g,n,v}^{3,4} + \rank d_{v,12}
  \\&=
  \rank d_{v,12} - \rank d_{v,2}.
\end{align*}
Putting everything together we obtain
\begin{equation}\label{equ:Dv ranks}
D_v =  \dim V_{g,n,v}^{3,4} - \rank d_{v,12} - \rank d_{v+1,12}
+ \rank d_{v+1,2}.
\end{equation}

Computationally, we generate bases of $V_{g,n,v}^{3,4}$ and $V_{g,n,v}^{4,6}$ with the help of nauty, and the matrices of $d_{v,2}$ and $d_{v,12}$, analogously to section \ref{sec:ordinary methods} above.
We then compute the ranks of the matrices $d_{v,2}$ and $d_{v,12}$ using LinBox.
In part, these computations are done modulo a prime (32189) to prevent memory problems.
Furthermore, the rank algorithm used by LinBox computes only a (likely exact) lower bound for the ranks.
In contrast to section \ref{sec:ordinary methods} we now cannot argue that our computation of $D_v$ is exact nevertheless, since the presence of the sign "+" in front of the last term of \eqref{equ:Dv ranks} spoils the argument. We shall hence consider the numbers depicted in Figure \ref{fig:merkulov} as (likely correct) approximations of $D_v$.

\subsection{Hairy graphs}

For the hairy graph complexex $\HG_{m,n}$, with the differential given by contracting edges, we found the dimensions of the homology shown in Figures \ref{fig:hairy_even_even}-\ref{fig:hairy_odd_odd}. In each case the shaded cells lie in the vanishing region of Proposition \ref{prop:HGC vanishing}.
We remark that parts of the data appeared in the literature before \cite{KWZ2,BarNatanMcKay}. Furthermore, for genera $\leq 2$ the homology is known analytically for any number of hairs \cite{CCTW}.

Let us also draw attention to one particular aspect of the tables.
There is a map of complexes $\HG_{m,n}^{1-hair}\to \G_n[1]$ by removing the single hair from a 1-hair graph.
It is well known (see e.g., \cite[Theorem 1]{TurchinWillwacher2017}) that this map induces a surjection on homology
$$
H_k(\HG_{m,n}^{1-hair})\twoheadrightarrow H_{k-1}(\G_n).
$$
Through Euler characteristic computations one can see that this map cannot be an isomorphism for sufficiently high loop orders, but no homology class in the kernel had been found.
By our numerical computations we can however deduce:

\begin{prop}
  The map $H_k(\HG_{m,n}^{1-hair})\to H_{k-1}(\G_n)$ is an isomorphism in loop orders $\leq 7$ and all $n,k$.
  In loop order 8 and $n$ even the map is not injective, with the first class in the kernel spanned by a linear combination of graphs with 8 loops and 12 vertices.
\end{prop}

\begin{figure}[H]
    \centering

\smallskip

1 hairs

\scalebox{ 0.75 }{
\begin{tabular}{|g|F|F|F|F|F|F|F|F|F|F|F|F|F|F|F|F|F|F|}
 \rowcolor{Gray}
\hline
l,v & 0 & 1 & 2 & 3 & 4 & 5 & 6 & 7 & 8 & 9 & 10 & 11 & 12 & 13 & 14 & 15 & 16 & 17\\ 
\hline
0 & - \cellcolor{LightGray} & - \cellcolor{LightGray} & - \cellcolor{LightGray} & - \cellcolor{LightGray} & - \cellcolor{LightGray} & - \cellcolor{LightGray} & - \cellcolor{LightGray} & - \cellcolor{LightGray} & - \cellcolor{LightGray} & - \cellcolor{LightGray} & - \cellcolor{LightGray} & - \cellcolor{LightGray} & - \cellcolor{LightGray} & - \cellcolor{LightGray} & - \cellcolor{LightGray} & - \cellcolor{LightGray} & - \cellcolor{LightGray} & - \cellcolor{LightGray}\\ 
1 & - & - & - \cellcolor{LightGray} & - \cellcolor{LightGray} & - \cellcolor{LightGray} & - \cellcolor{LightGray} & - \cellcolor{LightGray} & - \cellcolor{LightGray} & - \cellcolor{LightGray} & - \cellcolor{LightGray} & - \cellcolor{LightGray} & - \cellcolor{LightGray} & - \cellcolor{LightGray} & - \cellcolor{LightGray} & - \cellcolor{LightGray} & - \cellcolor{LightGray} & - \cellcolor{LightGray} & - \cellcolor{LightGray}\\ 
2 & - \cellcolor{LightGray} & - & - & - & - \cellcolor{LightGray} & - \cellcolor{LightGray} & - \cellcolor{LightGray} & - \cellcolor{LightGray} & - \cellcolor{LightGray} & - \cellcolor{LightGray} & - \cellcolor{LightGray} & - \cellcolor{LightGray} & - \cellcolor{LightGray} & - \cellcolor{LightGray} & - \cellcolor{LightGray} & - \cellcolor{LightGray} & - \cellcolor{LightGray} & - \cellcolor{LightGray}\\ 
3 & - \cellcolor{LightGray} & - \cellcolor{LightGray} & - & - & 1 & 0 & - \cellcolor{LightGray} & - \cellcolor{LightGray} & - \cellcolor{LightGray} & - \cellcolor{LightGray} & - \cellcolor{LightGray} & - \cellcolor{LightGray} & - \cellcolor{LightGray} & - \cellcolor{LightGray} & - \cellcolor{LightGray} & - \cellcolor{LightGray} & - \cellcolor{LightGray} & - \cellcolor{LightGray}\\
\hline
\end{tabular}

 }

\smallskip

2 hairs

\scalebox{ 0.75 }{
 \begin{tabular}{|g|F|F|F|F|F|F|F|F|F|F|F|F|F|F|F|F|F|F|}
 \rowcolor{Gray}
\hline
l,v & 0 & 1 & 2 & 3 & 4 & 5 & 6 & 7 & 8 & 9 & 10 & 11 & 12 & 13 & 14 & 15 & 16 & 17\\ 
\hline
0 & - & - \cellcolor{LightGray} & - \cellcolor{LightGray} & - \cellcolor{LightGray} & - \cellcolor{LightGray} & - \cellcolor{LightGray} & - \cellcolor{LightGray} & - \cellcolor{LightGray} & - \cellcolor{LightGray} & - \cellcolor{LightGray} & - \cellcolor{LightGray} & - \cellcolor{LightGray} & - \cellcolor{LightGray} & - \cellcolor{LightGray} & - \cellcolor{LightGray} & - \cellcolor{LightGray} & - \cellcolor{LightGray} & - \cellcolor{LightGray}\\ 
1 & - \cellcolor{LightGray} & - & - & - \cellcolor{LightGray} & - \cellcolor{LightGray} & - \cellcolor{LightGray} & - \cellcolor{LightGray} & - \cellcolor{LightGray} & - \cellcolor{LightGray} & - \cellcolor{LightGray} & - \cellcolor{LightGray} & - \cellcolor{LightGray} & - \cellcolor{LightGray} & - \cellcolor{LightGray} & - \cellcolor{LightGray} & - \cellcolor{LightGray} & - \cellcolor{LightGray} & - \cellcolor{LightGray}\\ 
2 & - \cellcolor{LightGray} & - \cellcolor{LightGray} & - & - & 0 & - \cellcolor{LightGray} & - \cellcolor{LightGray} & - \cellcolor{LightGray} & - \cellcolor{LightGray} & - \cellcolor{LightGray} & - \cellcolor{LightGray} & - \cellcolor{LightGray} & - \cellcolor{LightGray} & - \cellcolor{LightGray} & - \cellcolor{LightGray} & - \cellcolor{LightGray} & - \cellcolor{LightGray} & - \cellcolor{LightGray}\\ 
3 & - \cellcolor{LightGray} & - \cellcolor{LightGray} & - \cellcolor{LightGray} & - & 0 & 0 & 0 & - \cellcolor{LightGray} & - \cellcolor{LightGray} & - \cellcolor{LightGray} & - \cellcolor{LightGray} & - \cellcolor{LightGray} & - \cellcolor{LightGray} & - \cellcolor{LightGray} & - \cellcolor{LightGray} & - \cellcolor{LightGray} & - \cellcolor{LightGray} & - \cellcolor{LightGray}\\
\hline
\end{tabular}

 }

\smallskip

3 hairs

\scalebox{ 0.75 }{
 \begin{tabular}{|g|F|F|F|F|F|F|F|F|F|F|F|F|F|F|F|F|F|F|}
 \rowcolor{Gray}
\hline
l,v & 0 & 1 & 2 & 3 & 4 & 5 & 6 & 7 & 8 & 9 & 10 & 11 & 12 & 13 & 14 & 15 & 16 & 17\\ 
\hline
0 & - \cellcolor{LightGray} & - & - \cellcolor{LightGray} & - \cellcolor{LightGray} & - \cellcolor{LightGray} & - \cellcolor{LightGray} & - \cellcolor{LightGray} & - \cellcolor{LightGray} & - \cellcolor{LightGray} & - \cellcolor{LightGray} & - \cellcolor{LightGray} & - \cellcolor{LightGray} & - \cellcolor{LightGray} & - \cellcolor{LightGray} & - \cellcolor{LightGray} & - \cellcolor{LightGray} & - \cellcolor{LightGray} & - \cellcolor{LightGray}\\ 
1 & - \cellcolor{LightGray} & - \cellcolor{LightGray} & - & 1 & - \cellcolor{LightGray} & - \cellcolor{LightGray} & - \cellcolor{LightGray} & - \cellcolor{LightGray} & - \cellcolor{LightGray} & - \cellcolor{LightGray} & - \cellcolor{LightGray} & - \cellcolor{LightGray} & - \cellcolor{LightGray} & - \cellcolor{LightGray} & - \cellcolor{LightGray} & - \cellcolor{LightGray} & - \cellcolor{LightGray} & - \cellcolor{LightGray}\\ 
2 & - \cellcolor{LightGray} & - \cellcolor{LightGray} & - \cellcolor{LightGray} & - & 0 & 1 & - \cellcolor{LightGray} & - \cellcolor{LightGray} & - \cellcolor{LightGray} & - \cellcolor{LightGray} & - \cellcolor{LightGray} & - \cellcolor{LightGray} & - \cellcolor{LightGray} & - \cellcolor{LightGray} & - \cellcolor{LightGray} & - \cellcolor{LightGray} & - \cellcolor{LightGray} & - \cellcolor{LightGray}\\ 
3 & - \cellcolor{LightGray} & - \cellcolor{LightGray} & - \cellcolor{LightGray} & - \cellcolor{LightGray} & 0 & 0 & 0 & 2 & - \cellcolor{LightGray} & - \cellcolor{LightGray} & - \cellcolor{LightGray} & - \cellcolor{LightGray} & - \cellcolor{LightGray} & - \cellcolor{LightGray} & - \cellcolor{LightGray} & - \cellcolor{LightGray} & - \cellcolor{LightGray} & - \cellcolor{LightGray}\\
\hline
\end{tabular}

 }

\smallskip

4 hairs

\scalebox{ 0.75 }{
 \begin{tabular}{|g|F|F|F|F|F|F|F|F|F|F|F|F|F|F|F|F|F|F|}
 \rowcolor{Gray}
\hline
l,v & 0 & 1 & 2 & 3 & 4 & 5 & 6 & 7 & 8 & 9 & 10 & 11 & 12 & 13 & 14 & 15 & 16 & 17\\ 
\hline
0 & - \cellcolor{LightGray} & - \cellcolor{LightGray} & - & - \cellcolor{LightGray} & - \cellcolor{LightGray} & - \cellcolor{LightGray} & - \cellcolor{LightGray} & - \cellcolor{LightGray} & - \cellcolor{LightGray} & - \cellcolor{LightGray} & - \cellcolor{LightGray} & - \cellcolor{LightGray} & - \cellcolor{LightGray} & - \cellcolor{LightGray} & - \cellcolor{LightGray} & - \cellcolor{LightGray} & - \cellcolor{LightGray} & - \cellcolor{LightGray}\\ 
1 & - \cellcolor{LightGray} & - \cellcolor{LightGray} & - \cellcolor{LightGray} & - & 0 & - \cellcolor{LightGray} & - \cellcolor{LightGray} & - \cellcolor{LightGray} & - \cellcolor{LightGray} & - \cellcolor{LightGray} & - \cellcolor{LightGray} & - \cellcolor{LightGray} & - \cellcolor{LightGray} & - \cellcolor{LightGray} & - \cellcolor{LightGray} & - \cellcolor{LightGray} & - \cellcolor{LightGray} & - \cellcolor{LightGray}\\ 
2 & - \cellcolor{LightGray} & - \cellcolor{LightGray} & - \cellcolor{LightGray} & - \cellcolor{LightGray} & 0 & 0 & 0 & - \cellcolor{LightGray} & - \cellcolor{LightGray} & - \cellcolor{LightGray} & - \cellcolor{LightGray} & - \cellcolor{LightGray} & - \cellcolor{LightGray} & - \cellcolor{LightGray} & - \cellcolor{LightGray} & - \cellcolor{LightGray} & - \cellcolor{LightGray} & - \cellcolor{LightGray}\\ 
3 & - \cellcolor{LightGray} & - \cellcolor{LightGray} & - \cellcolor{LightGray} & - \cellcolor{LightGray} & 0 \cellcolor{LightGray} & 0 & 0 & 1 & 0 & - \cellcolor{LightGray} & - \cellcolor{LightGray} & - \cellcolor{LightGray} & - \cellcolor{LightGray} & - \cellcolor{LightGray} & - \cellcolor{LightGray} & - \cellcolor{LightGray} & - \cellcolor{LightGray} & - \cellcolor{LightGray}\\
\hline
\end{tabular}

 }

\smallskip

5 hairs

\scalebox{ 0.75 }{
 \begin{tabular}{|g|F|F|F|F|F|F|F|F|F|F|F|F|F|F|F|F|F|F|}
 \rowcolor{Gray}
\hline
l,v & 0 & 1 & 2 & 3 & 4 & 5 & 6 & 7 & 8 & 9 & 10 & 11 & 12 & 13 & 14 & 15 & 16 & 17\\ 
\hline
0 & - \cellcolor{LightGray} & - \cellcolor{LightGray} & - \cellcolor{LightGray} & - & - \cellcolor{LightGray} & - \cellcolor{LightGray} & - \cellcolor{LightGray} & - \cellcolor{LightGray} & - \cellcolor{LightGray} & - \cellcolor{LightGray} & - \cellcolor{LightGray} & - \cellcolor{LightGray} & - \cellcolor{LightGray} & - \cellcolor{LightGray} & - \cellcolor{LightGray} & - \cellcolor{LightGray} & - \cellcolor{LightGray} & - \cellcolor{LightGray}\\ 
1 & - \cellcolor{LightGray} & - \cellcolor{LightGray} & - \cellcolor{LightGray} & - \cellcolor{LightGray} & - & 0 & - \cellcolor{LightGray} & - \cellcolor{LightGray} & - \cellcolor{LightGray} & - \cellcolor{LightGray} & - \cellcolor{LightGray} & - \cellcolor{LightGray} & - \cellcolor{LightGray} & - \cellcolor{LightGray} & - \cellcolor{LightGray} & - \cellcolor{LightGray} & - \cellcolor{LightGray} & - \cellcolor{LightGray}\\ 
2 & - \cellcolor{LightGray} & - \cellcolor{LightGray} & - \cellcolor{LightGray} & - \cellcolor{LightGray} & - \cellcolor{LightGray} & 0 & 1 & 0 & - \cellcolor{LightGray} & - \cellcolor{LightGray} & - \cellcolor{LightGray} & - \cellcolor{LightGray} & - \cellcolor{LightGray} & - \cellcolor{LightGray} & - \cellcolor{LightGray} & - \cellcolor{LightGray} & - \cellcolor{LightGray} & - \cellcolor{LightGray}\\ 
3 & - \cellcolor{LightGray} & - \cellcolor{LightGray} & - \cellcolor{LightGray} & - \cellcolor{LightGray} & - \cellcolor{LightGray} & 0 \cellcolor{LightGray} & 0 & 0 & 2 & 0 & - \cellcolor{LightGray} & - \cellcolor{LightGray} & - \cellcolor{LightGray} & - \cellcolor{LightGray} & - \cellcolor{LightGray} & - \cellcolor{LightGray} & - \cellcolor{LightGray} & - \cellcolor{LightGray}\\
\hline
\end{tabular}

 }

\smallskip

6 hairs

\scalebox{ 0.75 }{
 \begin{tabular}{|g|F|F|F|F|F|F|F|F|F|F|F|F|F|F|F|F|F|F|}
 \rowcolor{Gray}
\hline
l,v & 0 & 1 & 2 & 3 & 4 & 5 & 6 & 7 & 8 & 9 & 10 & 11 & 12 & 13 & 14 & 15 & 16 & 17\\ 
\hline
0 & - \cellcolor{LightGray} & - \cellcolor{LightGray} & - \cellcolor{LightGray} & - \cellcolor{LightGray} & - & - \cellcolor{LightGray} & - \cellcolor{LightGray} & - \cellcolor{LightGray} & - \cellcolor{LightGray} & - \cellcolor{LightGray} & - \cellcolor{LightGray} & - \cellcolor{LightGray} & - \cellcolor{LightGray} & - \cellcolor{LightGray} & - \cellcolor{LightGray} & - \cellcolor{LightGray} & - \cellcolor{LightGray} & - \cellcolor{LightGray}\\ 
1 & - \cellcolor{LightGray} & - \cellcolor{LightGray} & - \cellcolor{LightGray} & - \cellcolor{LightGray} & - \cellcolor{LightGray} & - & 0 & - \cellcolor{LightGray} & - \cellcolor{LightGray} & - \cellcolor{LightGray} & - \cellcolor{LightGray} & - \cellcolor{LightGray} & - \cellcolor{LightGray} & - \cellcolor{LightGray} & - \cellcolor{LightGray} & - \cellcolor{LightGray} & - \cellcolor{LightGray} & - \cellcolor{LightGray}\\ 
2 & - \cellcolor{LightGray} & - \cellcolor{LightGray} & - \cellcolor{LightGray} & - \cellcolor{LightGray} & - \cellcolor{LightGray} & - \cellcolor{LightGray} & 0 & 0 & 0 & - \cellcolor{LightGray} & - \cellcolor{LightGray} & - \cellcolor{LightGray} & - \cellcolor{LightGray} & - \cellcolor{LightGray} & - \cellcolor{LightGray} & - \cellcolor{LightGray} & - \cellcolor{LightGray} & - \cellcolor{LightGray}\\ 
3 & - \cellcolor{LightGray} & - \cellcolor{LightGray} & - \cellcolor{LightGray} & - \cellcolor{LightGray} & - \cellcolor{LightGray} & - \cellcolor{LightGray} & 0 \cellcolor{LightGray} & 0 & 1 & 0 & 0 & - \cellcolor{LightGray} & - \cellcolor{LightGray} & - \cellcolor{LightGray} & - \cellcolor{LightGray} & - \cellcolor{LightGray} & - \cellcolor{LightGray} & - \cellcolor{LightGray}\\
\hline
\end{tabular}

 }

\smallskip

7 hairs

\scalebox{ 0.75 }{
 \begin{tabular}{|g|F|F|F|F|F|F|F|F|F|F|F|F|F|F|F|F|F|F|}
 \rowcolor{Gray}
\hline
l,v & 0 & 1 & 2 & 3 & 4 & 5 & 6 & 7 & 8 & 9 & 10 & 11 & 12 & 13 & 14 & 15 & 16 & 17\\ 
\hline
0 & - \cellcolor{LightGray} & - \cellcolor{LightGray} & - \cellcolor{LightGray} & - \cellcolor{LightGray} & - \cellcolor{LightGray} & - & - \cellcolor{LightGray} & - \cellcolor{LightGray} & - \cellcolor{LightGray} & - \cellcolor{LightGray} & - \cellcolor{LightGray} & - \cellcolor{LightGray} & - \cellcolor{LightGray} & - \cellcolor{LightGray} & - \cellcolor{LightGray} & - \cellcolor{LightGray} & - \cellcolor{LightGray} & - \cellcolor{LightGray}\\ 
1 & - \cellcolor{LightGray} & - \cellcolor{LightGray} & - \cellcolor{LightGray} & - \cellcolor{LightGray} & - \cellcolor{LightGray} & - \cellcolor{LightGray} & - & 1 & - \cellcolor{LightGray} & - \cellcolor{LightGray} & - \cellcolor{LightGray} & - \cellcolor{LightGray} & - \cellcolor{LightGray} & - \cellcolor{LightGray} & - \cellcolor{LightGray} & - \cellcolor{LightGray} & - \cellcolor{LightGray} & - \cellcolor{LightGray}\\ 
2 & - \cellcolor{LightGray} & - \cellcolor{LightGray} & - \cellcolor{LightGray} & - \cellcolor{LightGray} & - \cellcolor{LightGray} & - \cellcolor{LightGray} & - \cellcolor{LightGray} & 0 & 0 & 1 & - \cellcolor{LightGray} & - \cellcolor{LightGray} & - \cellcolor{LightGray} & - \cellcolor{LightGray} & - \cellcolor{LightGray} & - \cellcolor{LightGray} & - \cellcolor{LightGray} & - \cellcolor{LightGray}\\ 
3 & - \cellcolor{LightGray} & - \cellcolor{LightGray} & - \cellcolor{LightGray} & - \cellcolor{LightGray} & - \cellcolor{LightGray} & - \cellcolor{LightGray} & - \cellcolor{LightGray} & 0 \cellcolor{LightGray} & 0 & 0 & 0 & 2 & - \cellcolor{LightGray} & - \cellcolor{LightGray} & - \cellcolor{LightGray} & - \cellcolor{LightGray} & - \cellcolor{LightGray} & - \cellcolor{LightGray}\\
\hline
\end{tabular}

 }

\smallskip

8 hairs

\scalebox{ 0.75 }{
 \begin{tabular}{|g|F|F|F|F|F|F|F|F|F|F|F|F|F|F|F|F|F|F|}
 \rowcolor{Gray}
\hline
l,v & 0 & 1 & 2 & 3 & 4 & 5 & 6 & 7 & 8 & 9 & 10 & 11 & 12 & 13 & 14 & 15 & 16 & 17\\ 
\hline
0 & - \cellcolor{LightGray} & - \cellcolor{LightGray} & - \cellcolor{LightGray} & - \cellcolor{LightGray} & - \cellcolor{LightGray} & - \cellcolor{LightGray} & - & - \cellcolor{LightGray} & - \cellcolor{LightGray} & - \cellcolor{LightGray} & - \cellcolor{LightGray} & - \cellcolor{LightGray} & - \cellcolor{LightGray} & - \cellcolor{LightGray} & - \cellcolor{LightGray} & - \cellcolor{LightGray} & - \cellcolor{LightGray} & - \cellcolor{LightGray}\\ 
1 & - \cellcolor{LightGray} & - \cellcolor{LightGray} & - \cellcolor{LightGray} & - \cellcolor{LightGray} & - \cellcolor{LightGray} & - \cellcolor{LightGray} & - \cellcolor{LightGray} & - & 0 & - \cellcolor{LightGray} & - \cellcolor{LightGray} & - \cellcolor{LightGray} & - \cellcolor{LightGray} & - \cellcolor{LightGray} & - \cellcolor{LightGray} & - \cellcolor{LightGray} & - \cellcolor{LightGray} & - \cellcolor{LightGray}\\ 
2 & - \cellcolor{LightGray} & - \cellcolor{LightGray} & - \cellcolor{LightGray} & - \cellcolor{LightGray} & - \cellcolor{LightGray} & - \cellcolor{LightGray} & - \cellcolor{LightGray} & - \cellcolor{LightGray} & 0 & 0 & 0 & - \cellcolor{LightGray} & - \cellcolor{LightGray} & - \cellcolor{LightGray} & - \cellcolor{LightGray} & - \cellcolor{LightGray} & - \cellcolor{LightGray} & - \cellcolor{LightGray}\\ 
3 & - \cellcolor{LightGray} & - \cellcolor{LightGray} & - \cellcolor{LightGray} & - \cellcolor{LightGray} & - \cellcolor{LightGray} & - \cellcolor{LightGray} & - \cellcolor{LightGray} & - \cellcolor{LightGray} & 0 \cellcolor{LightGray} & 0 & 0 & 1 & 0 & - \cellcolor{LightGray} & - \cellcolor{LightGray} & - \cellcolor{LightGray} & - \cellcolor{LightGray} & - \cellcolor{LightGray}\\
\hline
\end{tabular}

 }
    \caption{Homology dimensions for $\HG_{m,n}$ with $n$ odd and $m$ even.}
    \label{fig:hairy_even_even}
\end{figure}

\begin{figure}[H]
    \centering

\smallskip

1 hairs

\scalebox{ 0.75 }{
 \begin{tabular}{|g|F|F|F|F|F|F|F|F|F|F|F|F|F|F|F|F|F|F|}
 \rowcolor{Gray}
\hline
l,v & 0 & 1 & 2 & 3 & 4 & 5 & 6 & 7 & 8 & 9 & 10 & 11 & 12 & 13 & 14 & 15 & 16 & 17\\ 
\hline
0 & - \cellcolor{LightGray} & - \cellcolor{LightGray} & - \cellcolor{LightGray} & - \cellcolor{LightGray} & - \cellcolor{LightGray} & - \cellcolor{LightGray} & - \cellcolor{LightGray} & - \cellcolor{LightGray} & - \cellcolor{LightGray} & - \cellcolor{LightGray} & - \cellcolor{LightGray} & - \cellcolor{LightGray} & - \cellcolor{LightGray} & - \cellcolor{LightGray} & - \cellcolor{LightGray} & - \cellcolor{LightGray} & - \cellcolor{LightGray} & - \cellcolor{LightGray}\\ 
1 & - & - & - \cellcolor{LightGray} & - \cellcolor{LightGray} & - \cellcolor{LightGray} & - \cellcolor{LightGray} & - \cellcolor{LightGray} & - \cellcolor{LightGray} & - \cellcolor{LightGray} & - \cellcolor{LightGray} & - \cellcolor{LightGray} & - \cellcolor{LightGray} & - \cellcolor{LightGray} & - \cellcolor{LightGray} & - \cellcolor{LightGray} & - \cellcolor{LightGray} & - \cellcolor{LightGray} & - \cellcolor{LightGray}\\ 
2 & - \cellcolor{LightGray} & - & - & - & - \cellcolor{LightGray} & - \cellcolor{LightGray} & - \cellcolor{LightGray} & - \cellcolor{LightGray} & - \cellcolor{LightGray} & - \cellcolor{LightGray} & - \cellcolor{LightGray} & - \cellcolor{LightGray} & - \cellcolor{LightGray} & - \cellcolor{LightGray} & - \cellcolor{LightGray} & - \cellcolor{LightGray} & - \cellcolor{LightGray} & - \cellcolor{LightGray}\\ 
3 & - \cellcolor{LightGray} & - \cellcolor{LightGray} & - & - & 1 & 0 & - \cellcolor{LightGray} & - \cellcolor{LightGray} & - \cellcolor{LightGray} & - \cellcolor{LightGray} & - \cellcolor{LightGray} & - \cellcolor{LightGray} & - \cellcolor{LightGray} & - \cellcolor{LightGray} & - \cellcolor{LightGray} & - \cellcolor{LightGray} & - \cellcolor{LightGray} & - \cellcolor{LightGray}\\
\hline
\end{tabular}

 }

\smallskip

2 hairs

\scalebox{ 0.75 }{
 \begin{tabular}{|g|F|F|F|F|F|F|F|F|F|F|F|F|F|F|F|F|F|F|}
 \rowcolor{Gray}
\hline
l,v & 0 & 1 & 2 & 3 & 4 & 5 & 6 & 7 & 8 & 9 & 10 & 11 & 12 & 13 & 14 & 15 & 16 & 17\\ 
\hline
0 & - & - \cellcolor{LightGray} & - \cellcolor{LightGray} & - \cellcolor{LightGray} & - \cellcolor{LightGray} & - \cellcolor{LightGray} & - \cellcolor{LightGray} & - \cellcolor{LightGray} & - \cellcolor{LightGray} & - \cellcolor{LightGray} & - \cellcolor{LightGray} & - \cellcolor{LightGray} & - \cellcolor{LightGray} & - \cellcolor{LightGray} & - \cellcolor{LightGray} & - \cellcolor{LightGray} & - \cellcolor{LightGray} & - \cellcolor{LightGray}\\ 
1 & - \cellcolor{LightGray} & - & - & - \cellcolor{LightGray} & - \cellcolor{LightGray} & - \cellcolor{LightGray} & - \cellcolor{LightGray} & - \cellcolor{LightGray} & - \cellcolor{LightGray} & - \cellcolor{LightGray} & - \cellcolor{LightGray} & - \cellcolor{LightGray} & - \cellcolor{LightGray} & - \cellcolor{LightGray} & - \cellcolor{LightGray} & - \cellcolor{LightGray} & - \cellcolor{LightGray} & - \cellcolor{LightGray}\\ 
2 & - \cellcolor{LightGray} & - \cellcolor{LightGray} & - & - & 1 & - \cellcolor{LightGray} & - \cellcolor{LightGray} & - \cellcolor{LightGray} & - \cellcolor{LightGray} & - \cellcolor{LightGray} & - \cellcolor{LightGray} & - \cellcolor{LightGray} & - \cellcolor{LightGray} & - \cellcolor{LightGray} & - \cellcolor{LightGray} & - \cellcolor{LightGray} & - \cellcolor{LightGray} & - \cellcolor{LightGray}\\ 
3 & - \cellcolor{LightGray} & - \cellcolor{LightGray} & - \cellcolor{LightGray} & - & 0 & 0 & 1 & - \cellcolor{LightGray} & - \cellcolor{LightGray} & - \cellcolor{LightGray} & - \cellcolor{LightGray} & - \cellcolor{LightGray} & - \cellcolor{LightGray} & - \cellcolor{LightGray} & - \cellcolor{LightGray} & - \cellcolor{LightGray} & - \cellcolor{LightGray} & - \cellcolor{LightGray}\\
\hline
\end{tabular}

 }

\smallskip

3 hairs

\scalebox{ 0.75 }{
 \begin{tabular}{|g|F|F|F|F|F|F|F|F|F|F|F|F|F|F|F|F|F|F|}
 \rowcolor{Gray}
\hline
l,v & 0 & 1 & 2 & 3 & 4 & 5 & 6 & 7 & 8 & 9 & 10 & 11 & 12 & 13 & 14 & 15 & 16 & 17\\ 
\hline
0 & - \cellcolor{LightGray} & - & - \cellcolor{LightGray} & - \cellcolor{LightGray} & - \cellcolor{LightGray} & - \cellcolor{LightGray} & - \cellcolor{LightGray} & - \cellcolor{LightGray} & - \cellcolor{LightGray} & - \cellcolor{LightGray} & - \cellcolor{LightGray} & - \cellcolor{LightGray} & - \cellcolor{LightGray} & - \cellcolor{LightGray} & - \cellcolor{LightGray} & - \cellcolor{LightGray} & - \cellcolor{LightGray} & - \cellcolor{LightGray}\\ 
1 & - \cellcolor{LightGray} & - \cellcolor{LightGray} & - & 0 & - \cellcolor{LightGray} & - \cellcolor{LightGray} & - \cellcolor{LightGray} & - \cellcolor{LightGray} & - \cellcolor{LightGray} & - \cellcolor{LightGray} & - \cellcolor{LightGray} & - \cellcolor{LightGray} & - \cellcolor{LightGray} & - \cellcolor{LightGray} & - \cellcolor{LightGray} & - \cellcolor{LightGray} & - \cellcolor{LightGray} & - \cellcolor{LightGray}\\ 
2 & - \cellcolor{LightGray} & - \cellcolor{LightGray} & - \cellcolor{LightGray} & - & 1 & 0 & - \cellcolor{LightGray} & - \cellcolor{LightGray} & - \cellcolor{LightGray} & - \cellcolor{LightGray} & - \cellcolor{LightGray} & - \cellcolor{LightGray} & - \cellcolor{LightGray} & - \cellcolor{LightGray} & - \cellcolor{LightGray} & - \cellcolor{LightGray} & - \cellcolor{LightGray} & - \cellcolor{LightGray}\\ 
3 & - \cellcolor{LightGray} & - \cellcolor{LightGray} & - \cellcolor{LightGray} & - \cellcolor{LightGray} & 0 & 0 & 1 & 0 & - \cellcolor{LightGray} & - \cellcolor{LightGray} & - \cellcolor{LightGray} & - \cellcolor{LightGray} & - \cellcolor{LightGray} & - \cellcolor{LightGray} & - \cellcolor{LightGray} & - \cellcolor{LightGray} & - \cellcolor{LightGray} & - \cellcolor{LightGray}\\
\hline
\end{tabular}

 }

\smallskip

4 hairs

\scalebox{ 0.75 }{
 \begin{tabular}{|g|F|F|F|F|F|F|F|F|F|F|F|F|F|F|F|F|F|F|}
 \rowcolor{Gray}
\hline
l,v & 0 & 1 & 2 & 3 & 4 & 5 & 6 & 7 & 8 & 9 & 10 & 11 & 12 & 13 & 14 & 15 & 16 & 17\\ 
\hline
0 & - \cellcolor{LightGray} & - \cellcolor{LightGray} & - & - \cellcolor{LightGray} & - \cellcolor{LightGray} & - \cellcolor{LightGray} & - \cellcolor{LightGray} & - \cellcolor{LightGray} & - \cellcolor{LightGray} & - \cellcolor{LightGray} & - \cellcolor{LightGray} & - \cellcolor{LightGray} & - \cellcolor{LightGray} & - \cellcolor{LightGray} & - \cellcolor{LightGray} & - \cellcolor{LightGray} & - \cellcolor{LightGray} & - \cellcolor{LightGray}\\ 
1 & - \cellcolor{LightGray} & - \cellcolor{LightGray} & - \cellcolor{LightGray} & - & 1 & - \cellcolor{LightGray} & - \cellcolor{LightGray} & - \cellcolor{LightGray} & - \cellcolor{LightGray} & - \cellcolor{LightGray} & - \cellcolor{LightGray} & - \cellcolor{LightGray} & - \cellcolor{LightGray} & - \cellcolor{LightGray} & - \cellcolor{LightGray} & - \cellcolor{LightGray} & - \cellcolor{LightGray} & - \cellcolor{LightGray}\\ 
2 & - \cellcolor{LightGray} & - \cellcolor{LightGray} & - \cellcolor{LightGray} & - \cellcolor{LightGray} & 0 & 0 & 1 & - \cellcolor{LightGray} & - \cellcolor{LightGray} & - \cellcolor{LightGray} & - \cellcolor{LightGray} & - \cellcolor{LightGray} & - \cellcolor{LightGray} & - \cellcolor{LightGray} & - \cellcolor{LightGray} & - \cellcolor{LightGray} & - \cellcolor{LightGray} & - \cellcolor{LightGray}\\ 
3 & - \cellcolor{LightGray} & - \cellcolor{LightGray} & - \cellcolor{LightGray} & - \cellcolor{LightGray} & 0 \cellcolor{LightGray} & 0 & 0 & 0 & 2 & - \cellcolor{LightGray} & - \cellcolor{LightGray} & - \cellcolor{LightGray} & - \cellcolor{LightGray} & - \cellcolor{LightGray} & - \cellcolor{LightGray} & - \cellcolor{LightGray} & - \cellcolor{LightGray} & - \cellcolor{LightGray}\\
\hline
\end{tabular}

 }

\smallskip

5 hairs

\scalebox{ 0.75 }{
 \begin{tabular}{|g|F|F|F|F|F|F|F|F|F|F|F|F|F|F|F|F|F|F|}
 \rowcolor{Gray}
\hline
l,v & 0 & 1 & 2 & 3 & 4 & 5 & 6 & 7 & 8 & 9 & 10 & 11 & 12 & 13 & 14 & 15 & 16 & 17\\ 
\hline
0 & - \cellcolor{LightGray} & - \cellcolor{LightGray} & - \cellcolor{LightGray} & - & - \cellcolor{LightGray} & - \cellcolor{LightGray} & - \cellcolor{LightGray} & - \cellcolor{LightGray} & - \cellcolor{LightGray} & - \cellcolor{LightGray} & - \cellcolor{LightGray} & - \cellcolor{LightGray} & - \cellcolor{LightGray} & - \cellcolor{LightGray} & - \cellcolor{LightGray} & - \cellcolor{LightGray} & - \cellcolor{LightGray} & - \cellcolor{LightGray}\\ 
1 & - \cellcolor{LightGray} & - \cellcolor{LightGray} & - \cellcolor{LightGray} & - \cellcolor{LightGray} & - & 0 & - \cellcolor{LightGray} & - \cellcolor{LightGray} & - \cellcolor{LightGray} & - \cellcolor{LightGray} & - \cellcolor{LightGray} & - \cellcolor{LightGray} & - \cellcolor{LightGray} & - \cellcolor{LightGray} & - \cellcolor{LightGray} & - \cellcolor{LightGray} & - \cellcolor{LightGray} & - \cellcolor{LightGray}\\ 
2 & - \cellcolor{LightGray} & - \cellcolor{LightGray} & - \cellcolor{LightGray} & - \cellcolor{LightGray} & - \cellcolor{LightGray} & 0 & 1 & 0 & - \cellcolor{LightGray} & - \cellcolor{LightGray} & - \cellcolor{LightGray} & - \cellcolor{LightGray} & - \cellcolor{LightGray} & - \cellcolor{LightGray} & - \cellcolor{LightGray} & - \cellcolor{LightGray} & - \cellcolor{LightGray} & - \cellcolor{LightGray}\\ 
3 & - \cellcolor{LightGray} & - \cellcolor{LightGray} & - \cellcolor{LightGray} & - \cellcolor{LightGray} & - \cellcolor{LightGray} & 0 \cellcolor{LightGray} & 0 & 0 & 3 & 0 & - \cellcolor{LightGray} & - \cellcolor{LightGray} & - \cellcolor{LightGray} & - \cellcolor{LightGray} & - \cellcolor{LightGray} & - \cellcolor{LightGray} & - \cellcolor{LightGray} & - \cellcolor{LightGray}\\
\hline
\end{tabular}

 }

\smallskip

6 hairs

\scalebox{ 0.75 }{
 \begin{tabular}{|g|F|F|F|F|F|F|F|F|F|F|F|F|F|F|F|F|F|F|}
 \rowcolor{Gray}
\hline
l,v & 0 & 1 & 2 & 3 & 4 & 5 & 6 & 7 & 8 & 9 & 10 & 11 & 12 & 13 & 14 & 15 & 16 & 17\\ 
\hline
0 & - \cellcolor{LightGray} & - \cellcolor{LightGray} & - \cellcolor{LightGray} & - \cellcolor{LightGray} & - & - \cellcolor{LightGray} & - \cellcolor{LightGray} & - \cellcolor{LightGray} & - \cellcolor{LightGray} & - \cellcolor{LightGray} & - \cellcolor{LightGray} & - \cellcolor{LightGray} & - \cellcolor{LightGray} & - \cellcolor{LightGray} & - \cellcolor{LightGray} & - \cellcolor{LightGray} & - \cellcolor{LightGray} & - \cellcolor{LightGray}\\ 
1 & - \cellcolor{LightGray} & - \cellcolor{LightGray} & - \cellcolor{LightGray} & - \cellcolor{LightGray} & - \cellcolor{LightGray} & - & 1 & - \cellcolor{LightGray} & - \cellcolor{LightGray} & - \cellcolor{LightGray} & - \cellcolor{LightGray} & - \cellcolor{LightGray} & - \cellcolor{LightGray} & - \cellcolor{LightGray} & - \cellcolor{LightGray} & - \cellcolor{LightGray} & - \cellcolor{LightGray} & - \cellcolor{LightGray}\\ 
2 & - \cellcolor{LightGray} & - \cellcolor{LightGray} & - \cellcolor{LightGray} & - \cellcolor{LightGray} & - \cellcolor{LightGray} & - \cellcolor{LightGray} & 0 & 0 & 2 & - \cellcolor{LightGray} & - \cellcolor{LightGray} & - \cellcolor{LightGray} & - \cellcolor{LightGray} & - \cellcolor{LightGray} & - \cellcolor{LightGray} & - \cellcolor{LightGray} & - \cellcolor{LightGray} & - \cellcolor{LightGray}\\ 
3 & - \cellcolor{LightGray} & - \cellcolor{LightGray} & - \cellcolor{LightGray} & - \cellcolor{LightGray} & - \cellcolor{LightGray} & - \cellcolor{LightGray} & 0 \cellcolor{LightGray} & 0 & 1 & 0 & 3 & - \cellcolor{LightGray} & - \cellcolor{LightGray} & - \cellcolor{LightGray} & - \cellcolor{LightGray} & - \cellcolor{LightGray} & - \cellcolor{LightGray} & - \cellcolor{LightGray}\\
\hline
\end{tabular}

 }

\smallskip

7 hairs

\scalebox{ 0.75 }{
 \begin{tabular}{|g|F|F|F|F|F|F|F|F|F|F|F|F|F|F|F|F|F|F|}
 \rowcolor{Gray}
\hline
l,v & 0 & 1 & 2 & 3 & 4 & 5 & 6 & 7 & 8 & 9 & 10 & 11 & 12 & 13 & 14 & 15 & 16 & 17\\ 
\hline
0 & - \cellcolor{LightGray} & - \cellcolor{LightGray} & - \cellcolor{LightGray} & - \cellcolor{LightGray} & - \cellcolor{LightGray} & - & - \cellcolor{LightGray} & - \cellcolor{LightGray} & - \cellcolor{LightGray} & - \cellcolor{LightGray} & - \cellcolor{LightGray} & - \cellcolor{LightGray} & - \cellcolor{LightGray} & - \cellcolor{LightGray} & - \cellcolor{LightGray} & - \cellcolor{LightGray} & - \cellcolor{LightGray} & - \cellcolor{LightGray}\\ 
1 & - \cellcolor{LightGray} & - \cellcolor{LightGray} & - \cellcolor{LightGray} & - \cellcolor{LightGray} & - \cellcolor{LightGray} & - \cellcolor{LightGray} & - & 0 & - \cellcolor{LightGray} & - \cellcolor{LightGray} & - \cellcolor{LightGray} & - \cellcolor{LightGray} & - \cellcolor{LightGray} & - \cellcolor{LightGray} & - \cellcolor{LightGray} & - \cellcolor{LightGray} & - \cellcolor{LightGray} & - \cellcolor{LightGray}\\ 
2 & - \cellcolor{LightGray} & - \cellcolor{LightGray} & - \cellcolor{LightGray} & - \cellcolor{LightGray} & - \cellcolor{LightGray} & - \cellcolor{LightGray} & - \cellcolor{LightGray} & 0 & 2 & 0 & - \cellcolor{LightGray} & - \cellcolor{LightGray} & - \cellcolor{LightGray} & - \cellcolor{LightGray} & - \cellcolor{LightGray} & - \cellcolor{LightGray} & - \cellcolor{LightGray} & - \cellcolor{LightGray}\\ 
3 & - \cellcolor{LightGray} & - \cellcolor{LightGray} & - \cellcolor{LightGray} & - \cellcolor{LightGray} & - \cellcolor{LightGray} & - \cellcolor{LightGray} & - \cellcolor{LightGray} & 0 \cellcolor{LightGray} & 0 & 0 & 5 & 0 & - \cellcolor{LightGray} & - \cellcolor{LightGray} & - \cellcolor{LightGray} & - \cellcolor{LightGray} & - \cellcolor{LightGray} & - \cellcolor{LightGray}\\
\hline
\end{tabular}

 }

\smallskip

8 hairs

\scalebox{ 0.75 }{
 \begin{tabular}{|g|F|F|F|F|F|F|F|F|F|F|F|F|F|F|F|F|F|F|}
 \rowcolor{Gray}
\hline
l,v & 0 & 1 & 2 & 3 & 4 & 5 & 6 & 7 & 8 & 9 & 10 & 11 & 12 & 13 & 14 & 15 & 16 & 17\\ 
\hline
0 & - \cellcolor{LightGray} & - \cellcolor{LightGray} & - \cellcolor{LightGray} & - \cellcolor{LightGray} & - \cellcolor{LightGray} & - \cellcolor{LightGray} & - & - \cellcolor{LightGray} & - \cellcolor{LightGray} & - \cellcolor{LightGray} & - \cellcolor{LightGray} & - \cellcolor{LightGray} & - \cellcolor{LightGray} & - \cellcolor{LightGray} & - \cellcolor{LightGray} & - \cellcolor{LightGray} & - \cellcolor{LightGray} & - \cellcolor{LightGray}\\ 
1 & - \cellcolor{LightGray} & - \cellcolor{LightGray} & - \cellcolor{LightGray} & - \cellcolor{LightGray} & - \cellcolor{LightGray} & - \cellcolor{LightGray} & - \cellcolor{LightGray} & - & 1 & - \cellcolor{LightGray} & - \cellcolor{LightGray} & - \cellcolor{LightGray} & - \cellcolor{LightGray} & - \cellcolor{LightGray} & - \cellcolor{LightGray} & - \cellcolor{LightGray} & - \cellcolor{LightGray} & - \cellcolor{LightGray}\\ 
2 & - \cellcolor{LightGray} & - \cellcolor{LightGray} & - \cellcolor{LightGray} & - \cellcolor{LightGray} & - \cellcolor{LightGray} & - \cellcolor{LightGray} & - \cellcolor{LightGray} & - \cellcolor{LightGray} & 0 & 0 & 2 & - \cellcolor{LightGray} & - \cellcolor{LightGray} & - \cellcolor{LightGray} & - \cellcolor{LightGray} & - \cellcolor{LightGray} & - \cellcolor{LightGray} & - \cellcolor{LightGray}\\ 
3 & - \cellcolor{LightGray} & - \cellcolor{LightGray} & - \cellcolor{LightGray} & - \cellcolor{LightGray} & - \cellcolor{LightGray} & - \cellcolor{LightGray} & - \cellcolor{LightGray} & - \cellcolor{LightGray} & 0 \cellcolor{LightGray} & 0 & 2 & 0 & 4 & - \cellcolor{LightGray} & - \cellcolor{LightGray} & - \cellcolor{LightGray} & - \cellcolor{LightGray} & - \cellcolor{LightGray}\\
\hline
\end{tabular}

 }
    \caption{Homology dimensions for $\HG_{m,n}$ with $n$ odd and $m$ odd.}
    \label{fig:hairy_even_odd}
\end{figure}

\begin{figure}[H]
    \centering

\smallskip

1 hairs

\scalebox{ 0.75 }{
 \begin{tabular}{|g|F|F|F|F|F|F|F|F|F|F|F|F|F|F|F|F|F|F|}
 \rowcolor{Gray}
\hline
l,v & 0 & 1 & 2 & 3 & 4 & 5 & 6 & 7 & 8 & 9 & 10 & 11 & 12 & 13 & 14 & 15 & 16 & 17\\ 
\hline
0 & - \cellcolor{LightGray} & - \cellcolor{LightGray} & - \cellcolor{LightGray} & - \cellcolor{LightGray} & - \cellcolor{LightGray} & - \cellcolor{LightGray} & - \cellcolor{LightGray} & - \cellcolor{LightGray} & - \cellcolor{LightGray} & - \cellcolor{LightGray} & - \cellcolor{LightGray} & - \cellcolor{LightGray} & - \cellcolor{LightGray} & - \cellcolor{LightGray} & - \cellcolor{LightGray} & - \cellcolor{LightGray} & - \cellcolor{LightGray} & - \cellcolor{LightGray}\\ 
1 & - & - & - \cellcolor{LightGray} & - \cellcolor{LightGray} & - \cellcolor{LightGray} & - \cellcolor{LightGray} & - \cellcolor{LightGray} & - \cellcolor{LightGray} & - \cellcolor{LightGray} & - \cellcolor{LightGray} & - \cellcolor{LightGray} & - \cellcolor{LightGray} & - \cellcolor{LightGray} & - \cellcolor{LightGray} & - \cellcolor{LightGray} & - \cellcolor{LightGray} & - \cellcolor{LightGray} & - \cellcolor{LightGray}\\ 
2 & - \cellcolor{LightGray} & - & - & - & - \cellcolor{LightGray} & - \cellcolor{LightGray} & - \cellcolor{LightGray} & - \cellcolor{LightGray} & - \cellcolor{LightGray} & - \cellcolor{LightGray} & - \cellcolor{LightGray} & - \cellcolor{LightGray} & - \cellcolor{LightGray} & - \cellcolor{LightGray} & - \cellcolor{LightGray} & - \cellcolor{LightGray} & - \cellcolor{LightGray} & - \cellcolor{LightGray}\\ 
3 & - \cellcolor{LightGray} & - \cellcolor{LightGray} & - & - & 1 & 0 & - \cellcolor{LightGray} & - \cellcolor{LightGray} & - \cellcolor{LightGray} & - \cellcolor{LightGray} & - \cellcolor{LightGray} & - \cellcolor{LightGray} & - \cellcolor{LightGray} & - \cellcolor{LightGray} & - \cellcolor{LightGray} & - \cellcolor{LightGray} & - \cellcolor{LightGray} & - \cellcolor{LightGray}\\
\hline
\end{tabular}

 }

\smallskip

2 hairs

\scalebox{ 0.75 }{
 \begin{tabular}{|g|F|F|F|F|F|F|F|F|F|F|F|F|F|F|F|F|F|F|}
 \rowcolor{Gray}
\hline
l,v & 0 & 1 & 2 & 3 & 4 & 5 & 6 & 7 & 8 & 9 & 10 & 11 & 12 & 13 & 14 & 15 & 16 & 17\\ 
\hline
0 & - & - \cellcolor{LightGray} & - \cellcolor{LightGray} & - \cellcolor{LightGray} & - \cellcolor{LightGray} & - \cellcolor{LightGray} & - \cellcolor{LightGray} & - \cellcolor{LightGray} & - \cellcolor{LightGray} & - \cellcolor{LightGray} & - \cellcolor{LightGray} & - \cellcolor{LightGray} & - \cellcolor{LightGray} & - \cellcolor{LightGray} & - \cellcolor{LightGray} & - \cellcolor{LightGray} & - \cellcolor{LightGray} & - \cellcolor{LightGray}\\ 
1 & - \cellcolor{LightGray} & - & - & - \cellcolor{LightGray} & - \cellcolor{LightGray} & - \cellcolor{LightGray} & - \cellcolor{LightGray} & - \cellcolor{LightGray} & - \cellcolor{LightGray} & - \cellcolor{LightGray} & - \cellcolor{LightGray} & - \cellcolor{LightGray} & - \cellcolor{LightGray} & - \cellcolor{LightGray} & - \cellcolor{LightGray} & - \cellcolor{LightGray} & - \cellcolor{LightGray} & - \cellcolor{LightGray}\\ 
2 & - \cellcolor{LightGray} & - \cellcolor{LightGray} & - & - & 0 & - \cellcolor{LightGray} & - \cellcolor{LightGray} & - \cellcolor{LightGray} & - \cellcolor{LightGray} & - \cellcolor{LightGray} & - \cellcolor{LightGray} & - \cellcolor{LightGray} & - \cellcolor{LightGray} & - \cellcolor{LightGray} & - \cellcolor{LightGray} & - \cellcolor{LightGray} & - \cellcolor{LightGray} & - \cellcolor{LightGray}\\ 
3 & - \cellcolor{LightGray} & - \cellcolor{LightGray} & - \cellcolor{LightGray} & - & 0 & 0 & 0 & - \cellcolor{LightGray} & - \cellcolor{LightGray} & - \cellcolor{LightGray} & - \cellcolor{LightGray} & - \cellcolor{LightGray} & - \cellcolor{LightGray} & - \cellcolor{LightGray} & - \cellcolor{LightGray} & - \cellcolor{LightGray} & - \cellcolor{LightGray} & - \cellcolor{LightGray}\\
\hline
\end{tabular}

 }

\smallskip

3 hairs

\scalebox{ 0.75 }{
 \begin{tabular}{|g|F|F|F|F|F|F|F|F|F|F|F|F|F|F|F|F|F|F|}
 \rowcolor{Gray}
\hline
l,v & 0 & 1 & 2 & 3 & 4 & 5 & 6 & 7 & 8 & 9 & 10 & 11 & 12 & 13 & 14 & 15 & 16 & 17\\ 
\hline
0 & - \cellcolor{LightGray} & - & - \cellcolor{LightGray} & - \cellcolor{LightGray} & - \cellcolor{LightGray} & - \cellcolor{LightGray} & - \cellcolor{LightGray} & - \cellcolor{LightGray} & - \cellcolor{LightGray} & - \cellcolor{LightGray} & - \cellcolor{LightGray} & - \cellcolor{LightGray} & - \cellcolor{LightGray} & - \cellcolor{LightGray} & - \cellcolor{LightGray} & - \cellcolor{LightGray} & - \cellcolor{LightGray} & - \cellcolor{LightGray}\\ 
1 & - \cellcolor{LightGray} & - \cellcolor{LightGray} & - & 1 & - \cellcolor{LightGray} & - \cellcolor{LightGray} & - \cellcolor{LightGray} & - \cellcolor{LightGray} & - \cellcolor{LightGray} & - \cellcolor{LightGray} & - \cellcolor{LightGray} & - \cellcolor{LightGray} & - \cellcolor{LightGray} & - \cellcolor{LightGray} & - \cellcolor{LightGray} & - \cellcolor{LightGray} & - \cellcolor{LightGray} & - \cellcolor{LightGray}\\ 
2 & - \cellcolor{LightGray} & - \cellcolor{LightGray} & - \cellcolor{LightGray} & - & 0 & 0 & - \cellcolor{LightGray} & - \cellcolor{LightGray} & - \cellcolor{LightGray} & - \cellcolor{LightGray} & - \cellcolor{LightGray} & - \cellcolor{LightGray} & - \cellcolor{LightGray} & - \cellcolor{LightGray} & - \cellcolor{LightGray} & - \cellcolor{LightGray} & - \cellcolor{LightGray} & - \cellcolor{LightGray}\\ 
3 & - \cellcolor{LightGray} & - \cellcolor{LightGray} & - \cellcolor{LightGray} & - \cellcolor{LightGray} & 0 & 0 & 0 & 0 & - \cellcolor{LightGray} & - \cellcolor{LightGray} & - \cellcolor{LightGray} & - \cellcolor{LightGray} & - \cellcolor{LightGray} & - \cellcolor{LightGray} & - \cellcolor{LightGray} & - \cellcolor{LightGray} & - \cellcolor{LightGray} & - \cellcolor{LightGray}\\
\hline
\end{tabular}

 }

\smallskip

4 hairs

\scalebox{ 0.75 }{
 \begin{tabular}{|g|F|F|F|F|F|F|F|F|F|F|F|F|F|F|F|F|F|F|}
 \rowcolor{Gray}
\hline
l,v & 0 & 1 & 2 & 3 & 4 & 5 & 6 & 7 & 8 & 9 & 10 & 11 & 12 & 13 & 14 & 15 & 16 & 17\\ 
\hline
0 & - \cellcolor{LightGray} & - \cellcolor{LightGray} & - & - \cellcolor{LightGray} & - \cellcolor{LightGray} & - \cellcolor{LightGray} & - \cellcolor{LightGray} & - \cellcolor{LightGray} & - \cellcolor{LightGray} & - \cellcolor{LightGray} & - \cellcolor{LightGray} & - \cellcolor{LightGray} & - \cellcolor{LightGray} & - \cellcolor{LightGray} & - \cellcolor{LightGray} & - \cellcolor{LightGray} & - \cellcolor{LightGray} & - \cellcolor{LightGray}\\ 
1 & - \cellcolor{LightGray} & - \cellcolor{LightGray} & - \cellcolor{LightGray} & - & 0 & - \cellcolor{LightGray} & - \cellcolor{LightGray} & - \cellcolor{LightGray} & - \cellcolor{LightGray} & - \cellcolor{LightGray} & - \cellcolor{LightGray} & - \cellcolor{LightGray} & - \cellcolor{LightGray} & - \cellcolor{LightGray} & - \cellcolor{LightGray} & - \cellcolor{LightGray} & - \cellcolor{LightGray} & - \cellcolor{LightGray}\\ 
2 & - \cellcolor{LightGray} & - \cellcolor{LightGray} & - \cellcolor{LightGray} & - \cellcolor{LightGray} & 0 & 0 & 0 & - \cellcolor{LightGray} & - \cellcolor{LightGray} & - \cellcolor{LightGray} & - \cellcolor{LightGray} & - \cellcolor{LightGray} & - \cellcolor{LightGray} & - \cellcolor{LightGray} & - \cellcolor{LightGray} & - \cellcolor{LightGray} & - \cellcolor{LightGray} & - \cellcolor{LightGray}\\ 
3 & - \cellcolor{LightGray} & - \cellcolor{LightGray} & - \cellcolor{LightGray} & - \cellcolor{LightGray} & 0 \cellcolor{LightGray} & 0 & 0 & 0 & 0 & - \cellcolor{LightGray} & - \cellcolor{LightGray} & - \cellcolor{LightGray} & - \cellcolor{LightGray} & - \cellcolor{LightGray} & - \cellcolor{LightGray} & - \cellcolor{LightGray} & - \cellcolor{LightGray} & - \cellcolor{LightGray}\\
\hline
\end{tabular}

 }

\smallskip

5 hairs

\scalebox{ 0.75 }{
 \begin{tabular}{|g|F|F|F|F|F|F|F|F|F|F|F|F|F|F|F|F|F|F|}
 \rowcolor{Gray}
\hline
l,v & 0 & 1 & 2 & 3 & 4 & 5 & 6 & 7 & 8 & 9 & 10 & 11 & 12 & 13 & 14 & 15 & 16 & 17\\ 
\hline
0 & - \cellcolor{LightGray} & - \cellcolor{LightGray} & - \cellcolor{LightGray} & - & - \cellcolor{LightGray} & - \cellcolor{LightGray} & - \cellcolor{LightGray} & - \cellcolor{LightGray} & - \cellcolor{LightGray} & - \cellcolor{LightGray} & - \cellcolor{LightGray} & - \cellcolor{LightGray} & - \cellcolor{LightGray} & - \cellcolor{LightGray} & - \cellcolor{LightGray} & - \cellcolor{LightGray} & - \cellcolor{LightGray} & - \cellcolor{LightGray}\\ 
1 & - \cellcolor{LightGray} & - \cellcolor{LightGray} & - \cellcolor{LightGray} & - \cellcolor{LightGray} & - & 1 & - \cellcolor{LightGray} & - \cellcolor{LightGray} & - \cellcolor{LightGray} & - \cellcolor{LightGray} & - \cellcolor{LightGray} & - \cellcolor{LightGray} & - \cellcolor{LightGray} & - \cellcolor{LightGray} & - \cellcolor{LightGray} & - \cellcolor{LightGray} & - \cellcolor{LightGray} & - \cellcolor{LightGray}\\ 
2 & - \cellcolor{LightGray} & - \cellcolor{LightGray} & - \cellcolor{LightGray} & - \cellcolor{LightGray} & - \cellcolor{LightGray} & 0 & 0 & 0 & - \cellcolor{LightGray} & - \cellcolor{LightGray} & - \cellcolor{LightGray} & - \cellcolor{LightGray} & - \cellcolor{LightGray} & - \cellcolor{LightGray} & - \cellcolor{LightGray} & - \cellcolor{LightGray} & - \cellcolor{LightGray} & - \cellcolor{LightGray}\\ 
3 & - \cellcolor{LightGray} & - \cellcolor{LightGray} & - \cellcolor{LightGray} & - \cellcolor{LightGray} & - \cellcolor{LightGray} & 0 \cellcolor{LightGray} & 0 & 0 & 0 & 0 & - \cellcolor{LightGray} & - \cellcolor{LightGray} & - \cellcolor{LightGray} & - \cellcolor{LightGray} & - \cellcolor{LightGray} & - \cellcolor{LightGray} & - \cellcolor{LightGray} & - \cellcolor{LightGray}\\
\hline
\end{tabular}

 }

\smallskip

6 hairs

\scalebox{ 0.75 }{
 \begin{tabular}{|g|F|F|F|F|F|F|F|F|F|F|F|F|F|F|F|F|F|F|}
 \rowcolor{Gray}
\hline
l,v & 0 & 1 & 2 & 3 & 4 & 5 & 6 & 7 & 8 & 9 & 10 & 11 & 12 & 13 & 14 & 15 & 16 & 17\\ 
\hline
0 & - \cellcolor{LightGray} & - \cellcolor{LightGray} & - \cellcolor{LightGray} & - \cellcolor{LightGray} & - & - \cellcolor{LightGray} & - \cellcolor{LightGray} & - \cellcolor{LightGray} & - \cellcolor{LightGray} & - \cellcolor{LightGray} & - \cellcolor{LightGray} & - \cellcolor{LightGray} & - \cellcolor{LightGray} & - \cellcolor{LightGray} & - \cellcolor{LightGray} & - \cellcolor{LightGray} & - \cellcolor{LightGray} & - \cellcolor{LightGray}\\ 
1 & - \cellcolor{LightGray} & - \cellcolor{LightGray} & - \cellcolor{LightGray} & - \cellcolor{LightGray} & - \cellcolor{LightGray} & - & 0 & - \cellcolor{LightGray} & - \cellcolor{LightGray} & - \cellcolor{LightGray} & - \cellcolor{LightGray} & - \cellcolor{LightGray} & - \cellcolor{LightGray} & - \cellcolor{LightGray} & - \cellcolor{LightGray} & - \cellcolor{LightGray} & - \cellcolor{LightGray} & - \cellcolor{LightGray}\\ 
2 & - \cellcolor{LightGray} & - \cellcolor{LightGray} & - \cellcolor{LightGray} & - \cellcolor{LightGray} & - \cellcolor{LightGray} & - \cellcolor{LightGray} & 0 & 0 & 1 & - \cellcolor{LightGray} & - \cellcolor{LightGray} & - \cellcolor{LightGray} & - \cellcolor{LightGray} & - \cellcolor{LightGray} & - \cellcolor{LightGray} & - \cellcolor{LightGray} & - \cellcolor{LightGray} & - \cellcolor{LightGray}\\ 
3 & - \cellcolor{LightGray} & - \cellcolor{LightGray} & - \cellcolor{LightGray} & - \cellcolor{LightGray} & - \cellcolor{LightGray} & - \cellcolor{LightGray} & 0 \cellcolor{LightGray} & 0 & 0 & 0 & 0 & - \cellcolor{LightGray} & - \cellcolor{LightGray} & - \cellcolor{LightGray} & - \cellcolor{LightGray} & - \cellcolor{LightGray} & - \cellcolor{LightGray} & - \cellcolor{LightGray}\\
\hline
\end{tabular}

 }

\smallskip

7 hairs

\scalebox{ 0.75 }{
 \begin{tabular}{|g|F|F|F|F|F|F|F|F|F|F|F|F|F|F|F|F|F|F|}
 \rowcolor{Gray}
\hline
l,v & 0 & 1 & 2 & 3 & 4 & 5 & 6 & 7 & 8 & 9 & 10 & 11 & 12 & 13 & 14 & 15 & 16 & 17\\ 
\hline
0 & - \cellcolor{LightGray} & - \cellcolor{LightGray} & - \cellcolor{LightGray} & - \cellcolor{LightGray} & - \cellcolor{LightGray} & - & - \cellcolor{LightGray} & - \cellcolor{LightGray} & - \cellcolor{LightGray} & - \cellcolor{LightGray} & - \cellcolor{LightGray} & - \cellcolor{LightGray} & - \cellcolor{LightGray} & - \cellcolor{LightGray} & - \cellcolor{LightGray} & - \cellcolor{LightGray} & - \cellcolor{LightGray} & - \cellcolor{LightGray}\\ 
1 & - \cellcolor{LightGray} & - \cellcolor{LightGray} & - \cellcolor{LightGray} & - \cellcolor{LightGray} & - \cellcolor{LightGray} & - \cellcolor{LightGray} & - & 1 & - \cellcolor{LightGray} & - \cellcolor{LightGray} & - \cellcolor{LightGray} & - \cellcolor{LightGray} & - \cellcolor{LightGray} & - \cellcolor{LightGray} & - \cellcolor{LightGray} & - \cellcolor{LightGray} & - \cellcolor{LightGray} & - \cellcolor{LightGray}\\ 
2 & - \cellcolor{LightGray} & - \cellcolor{LightGray} & - \cellcolor{LightGray} & - \cellcolor{LightGray} & - \cellcolor{LightGray} & - \cellcolor{LightGray} & - \cellcolor{LightGray} & 0 & 1 & 0 & - \cellcolor{LightGray} & - \cellcolor{LightGray} & - \cellcolor{LightGray} & - \cellcolor{LightGray} & - \cellcolor{LightGray} & - \cellcolor{LightGray} & - \cellcolor{LightGray} & - \cellcolor{LightGray}\\ 
3 & - \cellcolor{LightGray} & - \cellcolor{LightGray} & - \cellcolor{LightGray} & - \cellcolor{LightGray} & - \cellcolor{LightGray} & - \cellcolor{LightGray} & - \cellcolor{LightGray} & 0 \cellcolor{LightGray} & 0 & 0 & 0 & 0 & - \cellcolor{LightGray} & - \cellcolor{LightGray} & - \cellcolor{LightGray} & - \cellcolor{LightGray} & - \cellcolor{LightGray} & - \cellcolor{LightGray}\\
\hline
\end{tabular}

 }

\smallskip

8 hairs

\scalebox{ 0.75 }{
 \begin{tabular}{|g|F|F|F|F|F|F|F|F|F|F|F|F|F|F|F|F|F|F|}
 \rowcolor{Gray}
\hline
l,v & 0 & 1 & 2 & 3 & 4 & 5 & 6 & 7 & 8 & 9 & 10 & 11 & 12 & 13 & 14 & 15 & 16 & 17\\ 
\hline
0 & - \cellcolor{LightGray} & - \cellcolor{LightGray} & - \cellcolor{LightGray} & - \cellcolor{LightGray} & - \cellcolor{LightGray} & - \cellcolor{LightGray} & - & - \cellcolor{LightGray} & - \cellcolor{LightGray} & - \cellcolor{LightGray} & - \cellcolor{LightGray} & - \cellcolor{LightGray} & - \cellcolor{LightGray} & - \cellcolor{LightGray} & - \cellcolor{LightGray} & - \cellcolor{LightGray} & - \cellcolor{LightGray} & - \cellcolor{LightGray}\\ 
1 & - \cellcolor{LightGray} & - \cellcolor{LightGray} & - \cellcolor{LightGray} & - \cellcolor{LightGray} & - \cellcolor{LightGray} & - \cellcolor{LightGray} & - \cellcolor{LightGray} & - & 0 & - \cellcolor{LightGray} & - \cellcolor{LightGray} & - \cellcolor{LightGray} & - \cellcolor{LightGray} & - \cellcolor{LightGray} & - \cellcolor{LightGray} & - \cellcolor{LightGray} & - \cellcolor{LightGray} & - \cellcolor{LightGray}\\ 
2 & - \cellcolor{LightGray} & - \cellcolor{LightGray} & - \cellcolor{LightGray} & - \cellcolor{LightGray} & - \cellcolor{LightGray} & - \cellcolor{LightGray} & - \cellcolor{LightGray} & - \cellcolor{LightGray} & 0 & 0 & 1 & - \cellcolor{LightGray} & - \cellcolor{LightGray} & - \cellcolor{LightGray} & - \cellcolor{LightGray} & - \cellcolor{LightGray} & - \cellcolor{LightGray} & - \cellcolor{LightGray}\\ 
3 & - \cellcolor{LightGray} & - \cellcolor{LightGray} & - \cellcolor{LightGray} & - \cellcolor{LightGray} & - \cellcolor{LightGray} & - \cellcolor{LightGray} & - \cellcolor{LightGray} & - \cellcolor{LightGray} & 0 \cellcolor{LightGray} & 0 & 0 & 1 & 0 & - \cellcolor{LightGray} & - \cellcolor{LightGray} & - \cellcolor{LightGray} & - \cellcolor{LightGray} & - \cellcolor{LightGray}\\
\hline
\end{tabular}

 }
    \caption{Homology dimensions for $\HG_{m,n}$ with $n$ even and $m$ even.}
    \label{fig:hairy_odd_even}
\end{figure}

\begin{figure}[H]
    \centering

\smallskip

1 hairs

\scalebox{ 0.75 }{
 \begin{tabular}{|g|F|F|F|F|F|F|F|F|F|F|F|F|F|F|F|F|F|F|}
 \rowcolor{Gray}
\hline
l,v & 0 & 1 & 2 & 3 & 4 & 5 & 6 & 7 & 8 & 9 & 10 & 11 & 12 & 13 & 14 & 15 & 16 & 17\\ 
\hline
0 & - \cellcolor{LightGray} & - \cellcolor{LightGray} & - \cellcolor{LightGray} & - \cellcolor{LightGray} & - \cellcolor{LightGray} & - \cellcolor{LightGray} & - \cellcolor{LightGray} & - \cellcolor{LightGray} & - \cellcolor{LightGray} & - \cellcolor{LightGray} & - \cellcolor{LightGray} & - \cellcolor{LightGray} & - \cellcolor{LightGray} & - \cellcolor{LightGray} & - \cellcolor{LightGray} & - \cellcolor{LightGray} & - \cellcolor{LightGray} & - \cellcolor{LightGray}\\ 
1 & - & - & - \cellcolor{LightGray} & - \cellcolor{LightGray} & - \cellcolor{LightGray} & - \cellcolor{LightGray} & - \cellcolor{LightGray} & - \cellcolor{LightGray} & - \cellcolor{LightGray} & - \cellcolor{LightGray} & - \cellcolor{LightGray} & - \cellcolor{LightGray} & - \cellcolor{LightGray} & - \cellcolor{LightGray} & - \cellcolor{LightGray} & - \cellcolor{LightGray} & - \cellcolor{LightGray} & - \cellcolor{LightGray}\\ 
2 & - \cellcolor{LightGray} & - & - & - & - \cellcolor{LightGray} & - \cellcolor{LightGray} & - \cellcolor{LightGray} & - \cellcolor{LightGray} & - \cellcolor{LightGray} & - \cellcolor{LightGray} & - \cellcolor{LightGray} & - \cellcolor{LightGray} & - \cellcolor{LightGray} & - \cellcolor{LightGray} & - \cellcolor{LightGray} & - \cellcolor{LightGray} & - \cellcolor{LightGray} & - \cellcolor{LightGray}\\ 
3 & - \cellcolor{LightGray} & - \cellcolor{LightGray} & - & - & 1 & 0 & - \cellcolor{LightGray} & - \cellcolor{LightGray} & - \cellcolor{LightGray} & - \cellcolor{LightGray} & - \cellcolor{LightGray} & - \cellcolor{LightGray} & - \cellcolor{LightGray} & - \cellcolor{LightGray} & - \cellcolor{LightGray} & - \cellcolor{LightGray} & - \cellcolor{LightGray} & - \cellcolor{LightGray}\\
\hline
\end{tabular}

 }

\smallskip

2 hairs

\scalebox{ 0.75 }{
 \begin{tabular}{|g|F|F|F|F|F|F|F|F|F|F|F|F|F|F|F|F|F|F|}
 \rowcolor{Gray}
\hline
l,v & 0 & 1 & 2 & 3 & 4 & 5 & 6 & 7 & 8 & 9 & 10 & 11 & 12 & 13 & 14 & 15 & 16 & 17\\ 
\hline
0 & - & - \cellcolor{LightGray} & - \cellcolor{LightGray} & - \cellcolor{LightGray} & - \cellcolor{LightGray} & - \cellcolor{LightGray} & - \cellcolor{LightGray} & - \cellcolor{LightGray} & - \cellcolor{LightGray} & - \cellcolor{LightGray} & - \cellcolor{LightGray} & - \cellcolor{LightGray} & - \cellcolor{LightGray} & - \cellcolor{LightGray} & - \cellcolor{LightGray} & - \cellcolor{LightGray} & - \cellcolor{LightGray} & - \cellcolor{LightGray}\\ 
1 & - \cellcolor{LightGray} & - & - & - \cellcolor{LightGray} & - \cellcolor{LightGray} & - \cellcolor{LightGray} & - \cellcolor{LightGray} & - \cellcolor{LightGray} & - \cellcolor{LightGray} & - \cellcolor{LightGray} & - \cellcolor{LightGray} & - \cellcolor{LightGray} & - \cellcolor{LightGray} & - \cellcolor{LightGray} & - \cellcolor{LightGray} & - \cellcolor{LightGray} & - \cellcolor{LightGray} & - \cellcolor{LightGray}\\ 
2 & - \cellcolor{LightGray} & - \cellcolor{LightGray} & - & - & 1 & - \cellcolor{LightGray} & - \cellcolor{LightGray} & - \cellcolor{LightGray} & - \cellcolor{LightGray} & - \cellcolor{LightGray} & - \cellcolor{LightGray} & - \cellcolor{LightGray} & - \cellcolor{LightGray} & - \cellcolor{LightGray} & - \cellcolor{LightGray} & - \cellcolor{LightGray} & - \cellcolor{LightGray} & - \cellcolor{LightGray}\\ 
3 & - \cellcolor{LightGray} & - \cellcolor{LightGray} & - \cellcolor{LightGray} & - & 0 & 0 & 0 & - \cellcolor{LightGray} & - \cellcolor{LightGray} & - \cellcolor{LightGray} & - \cellcolor{LightGray} & - \cellcolor{LightGray} & - \cellcolor{LightGray} & - \cellcolor{LightGray} & - \cellcolor{LightGray} & - \cellcolor{LightGray} & - \cellcolor{LightGray} & - \cellcolor{LightGray}\\
\hline
\end{tabular}%

 }
    \caption{Homology dimensions for $\HG_{m,n}$ with $n$ even and $m$ odd.}
    \label{fig:hairy_odd_odd}
\end{figure}

\subsubsection{Methods}
The computation of the hairy graph homology is essentially identical to the computation of the ordinary graph homology $H(\G_n)$, and hence we just refer to section \ref{sec:ordinary methods} for algorithmic details.

As discussed there, a significant part of the rank computations were done in modular arithmetic. Nevertheless, table entries are upper bounds over $\Q$, and correct over $\Q$ if their two horizontally adjacent neighbors are zero.
Looking at the tables, this condition holds for most entries, though not all.

\subsection{Hairy graphs with distinguishable hairs}
The homology dimensions of the graph complexes with colored hairs are shown in Figures \ref{fig:chairy_odd} and \ref{fig:chairy_even}. At least for part of the entries we also show how the corresponding representation of the symmetric group $S_r$ (acting by permuting hair labels) decomposes into irreducible representations.

We note that a part of the data displayed in Figure \ref{fig:chairy_odd} has already appeared in \cite[Figure 1]{PayneWillwacher}. Furthermore, another subset of the numeric results had been found in \cite[Appendix A]{CGP2}.
Finally, for loop order 2, even $n$ and $r\leq 11$ hairs the homology has been computed in \cite{BCGY}.

\begin{figure}[H]
  \centering

\smallskip

2 hairs

\scalebox{ 1 }{


 }
  \caption{Homology dimensions for $\CHG_n$ with $n$ even.}
  \label{fig:chairy_odd}
\end{figure}

\begin{figure}[H]
  \centering

\smallskip

2 hairs

\scalebox{ .8 }{
 %

 }

  \caption{Homology dimensions for $\CHG_n$ with $n$ odd.}
  \label{fig:chairy_even}
\end{figure}

\subsubsection{Methods}\label{sec:methods chairy}
The homology dimensions are computed analogously to those of the ordinary graph complex, see section \ref{sec:ordinary methods}.

The decomposition into irreducible representations is computed as follows.
Let $V_{l,n,v,r}\subset \CHG_n(r)$ be the subspace spanned by graphs with $l$ loops and $v$ vertices.
For every irreducible representation of $S_r$, say $\lambda$, we compute the projector $P_{\lambda,v}$ onto the corresponding isotypical component of $V_{l,n,v,r}$.
Denote by $d_{v}: V_{l,n,v,r} \to V_{l,n,v-1,r}$ the contraction differential.
Then the dimension of the corresponding $\lambda$-isotypical component of the homology is 
\begin{equation}\label{equ:iso_rank}
\rank P_{\lambda,v} - \rank(d_{v}P_{\lambda,v})
-\rank(d_{v+1}P_{\lambda,v+1}).
\end{equation}
We note that our use of a matrix product $d_{v}P_{\lambda,v}$ is computationally expensive. This could in principle be avoided by using that 
\[
\ker d_v \mid_{\vimg P_{\lambda,v}} =
\ker \begin{pmatrix} d_v \\ 1-P_{\lambda,v} \end{pmatrix}. 
\]
We nevertheless restricted ourselves to the simplest implementation, using \eqref{equ:iso_rank}.

\subsection{Forested graph complex}

The dimensions of the homology groups of the forested graph complex $(\FG_n(r),d_c+d_u)$ are shown in Figures \ref{fig:forested_even} (for odd $n$) and \ref{fig:forested_odd} (for even $n$). In the tables, $l$ is the loop order and $m$ is the number of marked edges in the forest.
For some entries for $r\geq 2$ hairs we have also indicated the decomposition into irreducible components.

We note that the first table in Figure \ref{fig:forested_odd}, corresponding to the case of even $n$ and $r=0$ hairs, is not new.
It has previously been computed by Ohashi \cite{Ohashi} for $g\leq 6$ and Bartholdi \cite{Bartholdi} for g$\leq 7$.
We nevertheless include this table for reference and as confirmation of previous results.
The entries of loop order one are known for all $g$ \cite{ConantVogtmannAssembling}.
The other data in Figures \ref{fig:forested_even} (for odd $n$) and \ref{fig:forested_odd} are new to our knowledge.

Our tables do not contain loop order zero data. This is because $H(\FG_n^{0\lp})$ is well known and just a degree shifted version of the commutative operad. Furthermore, some computational simplifications due to Proposition \ref{prop:FGC dc bottom} only apply in positive loop order. 

\subsubsection{Morita classes and analogues for odd $n$}
We consider first the homology $H(\FGC_{0}^{2k+2}(0))$, see the first table (0 hairs) in Figure \ref{fig:forested_odd}.
In particular the entries in the $(l,m)=(4,4)$- and $(6,8)$-positions are known to be the first two of an infinite family of classes, the Morita classes $[\mu_{k}]\in H_{4k}(\FGC_{0}^{2k+2\lp}(0))$, see \cite{ConantVogtmannMorita, Morita1999}. They are spanned by the cycles of loop order $2k+2$ with $4k+2$ vertices and $4k$ marked edges
\[
    \mu_{k} 
    =
    \sum_{\sigma\in S_k}(-1)^\sigma
    \begin{tikzpicture}[scale=.6]
      \node[int] (v1) at (0,2) {};
      \node[int] (v2) at (0,1) {};
      \node[int] (v3) at (0,0) {};
      \node[int] (v4) at (0,-1) {};
      \node[int] (v5) at (0,-2) {};
      \node[int] (w1) at (3,2) {};
      \node[int] (w2) at (3,1) {};
      \node[int] (w3) at (3,0) {};
      \node[int] (w4) at (3,-1) {};
      \node[int] (w5) at (3,-2) {};
      \node at (1.5,3) {$\sigma$};
      \node[right] at (4.1,0) {$2k+1\times$};
      \draw (v2) edge[fatedge] (v1) edge[fatedge] (v3) edge (w2)
      (v1) edge[bend right] (v5) edge (w1)
      (v4) edge[fatedge] (v5) edge[fatedge] (v3) edge (w4)
      (v3) edge (w3)
      (v5) edge (w5)
      (w2) edge[fatedge] (w1) edge[fatedge] (w3)
      (w1) edge[bend left] (w5)
      (w4) edge[fatedge] (w5) edge[fatedge] (w3);
      \draw[dashed] (1.5,2.5) -- (1.5,-2.5);
      \draw[pbrace, thick] (4,2) -- (4,-2);
    \end{tikzpicture}
    .
\]
Here the dashed vertical line in the middle indicates that one connects the edges on the left- and right-hand sides according to the permutation $\sigma$.
The following conjecture is well known:

\begin{conj}\label{conj:morita}
The Morita homology classes $[\mu_{k}]\in H(\FGC_{0}^{2k+2}(0))$ are non-trivial for every $k\geq 1$.
\end{conj}

Next, consider the homology of $(\FGC_n^{g\text{-\lp}}(0),d_c+d_u)$ for odd $n$, see the first table of Figure \ref{fig:forested_even}.
Note the first non-zero entries at the positions $(l,m)=(4,3)$ and $(6,5)$.
Experimentally we found that the corresponding homology classes are represented by the forested graphs
\begin{align*}
  W_4 &:= \,
  \begin{tikzpicture}
  \node[int] (v0) at (0:1) {};
  \node[int] (v1) at (60:1) {};
  \node[int] (v2) at (120:1) {};
  \node[int] (v3) at (180:1) {};
  \node[int] (v4) at (240:1) {};
  \node[int] (v5) at (300:1) {};
  \draw (v1) edge[bend left] (v2)  edge[fatedge, bend right] (v2)
  (v2) edge (v3)
  (v3) edge[bend left] (v4)  edge[fatedge, bend right] (v4)
  (v4) edge (v5)
  (v5) edge[bend left] (v0)  edge[fatedge, bend right] (v0)
  (v0) edge (v1);
  \end{tikzpicture}
  \in \FGC_1^{4\lp}(0)
  &
  W_6 &:= \,
  \begin{tikzpicture}
  \node[int] (v0) at (0:1) {};
  \node[int] (v1) at (36:1) {};
  \node[int] (v2) at (72:1) {};
  \node[int] (v3) at (108:1) {};
  \node[int] (v4) at (144:1) {};
  \node[int] (v5) at (180:1) {};
  \node[int] (v6) at (216:1) {};
  \node[int] (v7) at (252:1) {};
  \node[int] (v8) at (288:1) {};
  \node[int] (v9) at (324:1) {};
  \draw (v1) edge[bend left] (v2)  edge[fatedge, bend right] (v2)
  (v2) edge (v3)
  (v3) edge[bend left] (v4)  edge[fatedge, bend right] (v4)
  (v4) edge (v5)
  (v5) edge[bend left] (v6)  edge[fatedge, bend right] (v6)
  (v6) edge (v7)
  (v5) edge[bend left] (v6)  edge[fatedge, bend right] (v6)
  (v6) edge (v7)
  (v7) edge[bend left] (v8)  edge[fatedge, bend right] (v8)
  (v8) edge (v9)
  (v9) edge[bend left] (v0)  edge[fatedge, bend right] (v0)
  (v0) edge (v1)
  ;
  \end{tikzpicture}
  \in \FGC_1^{6\lp}(0).
\end{align*}
They are also part of an infinite family of cocylces: Let us define
\[
W_{2k} := \,
\begin{tikzpicture}
\node (v0) at (0:1) {$\vdots$};
\node[int] (v1) at (45:1) {};
\node[int] (v2) at (90:1) {};
\node[int] (v3) at (135:1) {};
\node[int] (v4) at (180:1) {};
\node[int] (v5) at (225:1) {};
\node[int] (v6) at (270:1) {};
\node[int] (v7) at (315:1) {};
\draw (v1) edge[bend left] (v2)  edge[fatedge, bend right] (v2)
(v2) edge (v3)
(v3) edge[bend left] (v4)  edge[fatedge, bend right] (v4)
(v4) edge (v5)
(v5) edge[bend left] (v6)  edge[fatedge, bend right] (v6)
(v6) edge (v7);
\end{tikzpicture}
\quad\quad\quad 
\text{$4k-2$ vertices, $6k-3$ edges}.
\]

Then we have:
\begin{lemma}
The elements $W_{2k}\in \FGC_1^{2k\lp}(0)$ are non-zero and satisfy
\[
  d_c W_{2k} = d_u W_{2k} = 0. 
\]
\end{lemma}
\begin{proof}
The symmetry group of $W_{2k}$ is the dihedral group $D_{2k-1}$, and one readily verifies that the rotation and reflection generators of $D_{2k-1}$ act trivially on one and hence on any given orientation of the graph.
Hence $W_{2k}\neq 0$ in $\FGC_1^{2k\lp}(0)$.

Next, $d_c W_{2k} = 0$ holds trivially because the graphs resulting from the edge contraction all have a tadpole, and are hence zero.
Similarly, $d_u W_{2k}=0$ since unmarking of edges leads to graphs with multiple unmarked edges, and hence an odd symmetry by transposing a pair of such edges.
\end{proof}

We are hence led to the following conjecture, which can be seen as an "odd" analogue of Conjecture \ref{conj:morita}:

\begin{conj}\label{conj:morita new}
For each $k=2,3,4,\dots$ the elements $W_{2k}\in \FGC_1^{2k\lp}(0)$ represent non-trivial homology classes.
\end{conj}

Note that the definition of $W_{2k}$ above would also makes sense if one replaced $2k$ by an odd number. However, the graphs thus obtained are zero elements in the graph complex $\FGC_1(0)$, since they have an odd symmetry.

\subsubsection{Methods}
We compute the homology of the complex $(K_n,d_u)$, see Corollary \ref{cor:Kn}. 
To this end, we consider the finite dimensional graded subspaces 
\begin{align*}
  V_{n,g,m,r,e} & \subset  \FGC_n^{g\lp, e-exc}(r),
  &
V_{n,g,m,r} & \subset  \FGC_n^{g\lp, 0-exc}(r),
&
V_{n,g,m,r}' & \subset  \FGC_n^{g\lp, 1-exc}(r)
\end{align*}
with a fixed number of marked edges $m$. We have the operators
\begin{align*}
d_{c,g,m,r} :   V_{n,g,m,r} &\to V_{n,g,m-1,r}' \\
d_{u,g,m,r} : V_{n,g,m,r} &\to V_{n,g,m-1,r} \\
d_{uc,g,m,r}= \begin{pmatrix}d_{u,g,m,r} \\ d_{c,g,m,r}\end{pmatrix} : V_{n,g,m,r} &\to V_{n,g,m-1,r} \oplus V_{n,g,m-1,r}'
\end{align*}

We desire to compute the homology dimensions
\[
D_{g,m,r} = \dim \ker d_{c,g,m,r} 
- \rank d_{u,g,m,r}\mid_{\ker d_{c,g,m,r}}
- \rank d_{u,g,m+1,r}\mid_{\ker d_{c,g,m+1,r}}\, .
\]
Analogously to the derivation in section \ref{sec:ordinary me methods} we can rewrite this as 
\begin{equation}\label{equ:Dgmr}
D_{g,m,r} = \dim V_{n,g,m,r} 
- \rank d_{uc,g,m,r}
- \rank d_{uc,g,m+1,r}
+ \rank d_{c,g,m+1,r} \, .
\end{equation}
By Proposition \ref{prop:FGC dc bottom} we may furthermore compute the last summand using an Euler characteristic,
\begin{equation}\label{equ:dcgmr}
  \rank d_{c,g,m+1,r}
  = 
  \sum_{e\geq 1} (-1)^{e-1} \dim V_{n,g,m+1-e,r,e}.
\end{equation}

We compute bases of the vector spaces $V_{n,g,m,r,e}$ by first producing a list of isomorphism of hairy graphs using nauty, and then iterating over all possible choices of edge markings.
We discard graphs that are not bridgeless.
Then we build matrices of the operators $d_{c,g,m,r}$, $d_{u,g,m,r}$ and $d_{uc,g,m,r}$.
We compute the ranks of the matrices $d_{uc,g,m,r}$ using LinBox, partially over a finite field.
As noted in section \ref{sec:ordinary methods} the ranks thus produced are only lower bounds for the true (rational) ranks, both because we use prime arithmetic and the algorithm used by LinBox. 
However, note that the formula \eqref{equ:dcgmr} gives the exact rank of $d_{c,g,m+1,r}$, and the two other ranks appearing in \eqref{equ:Dgmr} come with a minus sign.
Hence we may use the argument of section \ref{sec:ordinary methods} to conclude that the computed homology dimensions are upper bounds to the true (rational) homology dimensions. They are furthermore correct over $\mathbb Q$ as long as every involved operator $d_{c,g,m,r}$ contributes to one computation with zero result, that is, if the horizontally adjacent table entries are both zero.
Again this holds for many, though not all entries in the tables of Figures \ref{fig:forested_even} and \ref{fig:forested_odd}.

We also remark that the simplifications of Propositions \ref{prop:FGC bl1} and \ref{prop:FGC dc bottom} and Corollary \ref{cor:Kn} help the computation.
To illustrate the rough orders of magnitude, we tabulate here the dimensions of some underlying graph complexes.
\begin{center}
\begin{tabular}{c|ccccc}
   $m$ & 4 & 5 & 6 & 7 & 8 \\
   \hline
   $\dim \FGC_0^{m,bl,7\lp}(0)$ & 
   1961033 & 4773384 & 8382482 & 10608362 & 9443274
   \\
   $\dim \FGC_0^{m,0\exc, 7\lp}(0)$ & 119751 & 344107  &749292 &1252757 &1594769\\
   $\dim \FGC_0^{m,bl,0\exc, 7\lp}(0)$ & 
   112474 & 322855 & 703262 & 1178567 & 1507269 \\
\end{tabular}
\end{center}
One sees that passing to the 0-excess part reduces the dimension of the vector spaces involved by an order of magnitude, while removing the bridgeless graphs yields only rather marginal improvements  ($<10\%$).

\begin{figure}[H]
  \centering

\smallskip

 0 hairs 

\scalebox{ 1 }{
 \begin{tabular}{|g|D|D|D|D|D|D|D|D|D|D|D|D|}
 \rowcolor{Gray}
\hline
l,m & 0 & 1 & 2 & 3 & 4 & 5 & 6 & 7 & 8 & 9 & 10 & 11\\ 
\hline
1 & - & - & - & - & - & - & - & - & - & - & - & -\\ 
2 & 0 & 0 & - & - & - & - & - & - & - & - & - & -\\ 
3 & 0 & 0 & 0 & 0 & - & - & - & - & - & - & - & -\\ 
4 & 0 & 0 & 0 & 1 & 0 & 0 & - & - & - & - & - & -\\ 
5 & 0 & 0 & 0 & 0 & 0 & 0 & 0 & 0 & - & - & - & -\\ 
6 & 0 & 0 & 0 & 0 & 0 & 1 & 0 & 0 & 0 & 0 & - & -\\ 
7 & 0  & 0 & 0 & 0 & 0 & 0 & 0  & 0  & 0  & 0  & 0  & 2\\ 
\hline
\end{tabular}

 }

\smallskip

 1 hairs 

\scalebox{ 1 }{
 \begin{tabular}{|g|D|D|D|D|D|D|D|D|D|D|D|D|}
 \rowcolor{Gray}
\hline
l,m & 0 & 1 & 2 & 3 & 4 & 5 & 6 & 7 & 8 & 9 & 10 & 11\\ 
\hline
1 & 0 & - & - & - & - & - & - & - & - & - & - & -\\ 
2 & 0 & 0 & 0 & - & - & - & - & - & - & - & - & -\\ 
3 & 0 & 0 & 0 & 1 & 0 & - & - & - & - & - & - & -\\ 
4 & 0 & 0 & 0 & 1 & 0 & 0 & 0 & - & - & - & - & -\\ 
5 & 0 & 0 & 0 & 0 & 0 & 1 & 0 & 0 & 0 & - & - & -\\ 
6 & 0  & 0  & 0  & 0 & 0 & 1  & 0  & 0  & 0  & 0  & 2 & -\\ 
\hline
\end{tabular}

 }

\smallskip

 2 hairs 

\scalebox{ 1 }{
 \begin{tabular}{|g|D|D|D|D|D|D|D|D|D|D|D|}
 \rowcolor{Gray}
\hline
l,m & 0 & 1 & 2 & 3 & 4 & 5 & 6 & 7 & 8 & 9 & 10\\ 
\hline
1 & 0 & 1 ($ s_{[1, 1]}$) & - & - & - & - & - & - & - & - & -\\ 
2 & 0 & 0 & 1 ($ s_{[2]}$) & 1 ($ s_{[1, 1]}$) & - & - & - & - & - & - & -\\ 
3 & 0 & 0 & 0 & 3 ($ s_{[2]}$ + $2s_{[1, 1]}$) & 0 & 0 & - & - & - & - & -\\ 
4 & 0 & 0 & 0 & 1 ($ s_{[2]}$) & 1 ($ s_{[2]}$) & 1 ($ s_{[1, 1]}$) & 0 & 2 ($ s_{[2]}$ + $s_{[1, 1]}$) & - & - & -\\ 
5 & 0  & 0  & 0 & 0 & 0 & 3 & 0 & 0 & 1 & 3 & -\\ 
\hline
\end{tabular}

 }

\smallskip

 3 hairs 

\scalebox{ 1 }{
 \begin{tabular}{|g|D|D|D|D|D|D|D|D|D|D|}
 \rowcolor{Gray}
\hline
l,m & 0 & 1 & 2 & 3 & 4 & 5 & 6 & 7 & 8 & 9\\ 
\hline
1 & 0 & 2 ($ s_{[2, 1]}$) & 0 & - & - & - & - & - & - & -\\ 
2 & 0 & 0 & 3 ($ s_{[3]}$ + $s_{[2, 1]}$) & 3 ($ s_{[2, 1]}$ + $s_{[1, 1, 1]}$) & 0 & - & - & - & - & -\\ 
3 & 0 & 0 & 0 & 7 ($ 2s_{[3]}$ + $2s_{[2, 1]}$ + $s_{[1, 1, 1]}$) & 0 & 1 ($ s_{[1, 1, 1]}$) & 0 & - & - & -\\ 
4 & 0  & 0  & 0  & 1 & 5 & 3 & 0 & 8  & 0 & -\\ 
\hline
\end{tabular}

 }

\smallskip

 4 hairs 

\scalebox{ 0.75 }{
 \begin{tabular}{|g|D|D|D|D|D|D|D|M|D|D|}
 \rowcolor{Gray}
\hline
l,m & 0 & 1 & 2 & 3 & 4 & 5 & 6 & 7 & 8 & 9\\ 
\hline
1 & 0 & 3 ($ s_{[3, 1]}$) & 0 & 1 ($ s_{[1, 1, 1, 1]}$) & - & - & - & - & - & -\\ 
2 & 0 & 0 & 6 ($ s_{[4]}$ + $s_{[3, 1]}$ + $s_{[2, 2]}$) & 6 ($ s_{[3, 1]}$ + $s_{[2, 1, 1]}$) & 0 & 1 ($ s_{[1, 1, 1, 1]}$) & - & - & - & -\\ 
3 & 0  & 0  & 0  & 14 ($ 2s_{[4]}$ + $3s_{[3, 1]}$ + $s_{[2, 1, 1]}$) & 0 & 4  ($ s_{[2, 1, 1]}$ + $s_{[1, 1, 1, 1]}$) & 0  & 6 ($ s_{[2, 2]}$ + $s_{[2, 1, 1]}$ + $s_{[1, 1, 1, 1]}$) & - & -\\ 
\hline
\end{tabular}

 }

\smallskip

 5 hairs 

\scalebox{ 1 }{
 \begin{tabular}{|g|D|D|D|D|D|M|D|D|D|D|}
 \rowcolor{Gray}
\hline
l,m & 0 & 1 & 2 & 3 & 4 & 5 & 6 & 7 & 8 & 9\\ 
\hline
1 & 0 & 4 ($ s_{[4, 1]}$) & 0 & 4 ($ s_{[2, 1, 1, 1]}$) & 0 & - & - & - & - & -\\ 
2 & 0 & 0 & 10 ($ s_{[5]}$ + $s_{[4, 1]}$ + $s_{[3, 2]}$) & 10 ($ s_{[4, 1]}$ + $s_{[3, 1, 1]}$) & 0 & 5 ($ s_{[2, 1, 1, 1]}$ + $s_{[1, 1, 1, 1, 1]}$) & 0 & - & - & -\\ 
3 & 0  & 0 & 0 & 25  & 0  & 10  & 0  & 36 & 0 & -\\ 
\hline
\end{tabular}

 }
  \caption{Homology dimensions for $\FG_n$ with $n$ odd.}
  \label{fig:forested_even}
\end{figure}

\begin{figure}[H]
  \centering

\smallskip

 0 hairs 

\scalebox{ 1 }{
 \begin{tabular}{|g|D|D|D|D|D|D|D|D|D|D|D|D|}
 \rowcolor{Gray}
\hline
l,m & 0 & 1 & 2 & 3 & 4 & 5 & 6 & 7 & 8 & 9 & 10 & 11\\ 
\hline
1 & - & - & - & - & - & - & - & - & - & - & - & -\\ 
2 & 1 & 0 & - & - & - & - & - & - & - & - & - & -\\ 
3 & 1 & 0 & 0 & 0 & - & - & - & - & - & - & - & -\\ 
4 & 1 & 0 & 0 & 0 & 1 & 0 & - & - & - & - & - & -\\ 
5 & 1 & 0 & 0 & 0 & 0 & 0 & 0 & 0 & - & - & - & -\\ 
6 & 1 & 0 & 0 & 0 & 0 & 0 & 0 & 0 & 1 & 0 & - & -\\ 
7 & 1  & 0  & 0  & 0  & 0  & 0  & 0  & 0  & 1  & 0  & 0  & 1 \\ 
\hline
\end{tabular}

 }

\smallskip

 1 hairs 

\scalebox{ 1 }{
 \begin{tabular}{|g|D|D|D|D|D|D|D|D|D|D|D|D|}
 \rowcolor{Gray}
\hline
l,m & 0 & 1 & 2 & 3 & 4 & 5 & 6 & 7 & 8 & 9 & 10 & 11\\ 
\hline
1 & 1 & - & - & - & - & - & - & - & - & - & - & -\\ 
2 & 1 & 0 & 0 & - & - & - & - & - & - & - & - & -\\ 
3 & 1 & 0 & 0 & 0 & 0 & - & - & - & - & - & - & -\\ 
4 & 1 & 0 & 0 & 0 & 1 & 0 & 0 & - & - & - & - & -\\ 
5 & 1 & 0 & 0 & 0 & 0 & 0 & 0 & 1 & 0 & - & - & -\\ 
6 & 1  & 0  & 0  & 0  & 0  & 0  & 0  & 0  & 1  & 0  & 1  & -\\ 
\hline
\end{tabular}

 }

\smallskip

 2 hairs 

\scalebox{ 1 }{
 \begin{tabular}{|g|D|D|D|D|D|D|D|D|D|D|D|}
 \rowcolor{Gray}
\hline
l,m & 0 & 1 & 2 & 3 & 4 & 5 & 6 & 7 & 8 & 9 & 10\\ 
\hline
1 & 1 ($ s_{[2]}$) & 0 & - & - & - & - & - & - & - & - & -\\ 
2 & 1 ($ s_{[2]}$) & 0 & 0 & 0 & - & - & - & - & - & - & -\\ 
3 & 1 ($ s_{[2]}$) & 0 & 0 & 0 & 1 ($ s_{[2]}$) & 0 & - & - & - & - & -\\ 
4 & 1 ($ s_{[2]}$) & 0 & 0 & 0 & 1 ($ s_{[2]}$) & 0 & 1 ($ s_{[1, 1]}$) & 0 & - & - & -\\ 
5 & 1 & 0 & 0 & 0 & 0 & 0 & 0 & 3 & 1 & 1 & -\\ 
\hline
\end{tabular}

 }

\smallskip

 3 hairs 

\scalebox{ 1 }{
 \begin{tabular}{|g|D|D|D|D|D|D|D|D|D|D|}
 \rowcolor{Gray}
\hline
l,m & 0 & 1 & 2 & 3 & 4 & 5 & 6 & 7 & 8 & 9\\ 
\hline
1 & 1 ($ s_{[3]}$) & 0 & 1 ($ s_{[1, 1, 1]}$) & - & - & - & - & - & - & -\\ 
2 & 1 ($ s_{[3]}$) & 0 & 0 & 0 & 0 & - & - & - & - & -\\ 
3 & 1 ($ s_{[3]}$) & 0 & 0 & 0 & 3 ($ s_{[3]}$ + $s_{[2, 1]}$) & 0 & 2 ($ s_{[2, 1]}$) & - & - & -\\ 
4 & 1 & 0 & 0 & 0 & 1 & 0 & 4 & 0 & 0 & -\\ 
\hline
\end{tabular}

 }

\smallskip

 4 hairs 

\scalebox{ 1 }{
 \begin{tabular}{|g|D|D|D|D|D|D|M|D|D|D|}
 \rowcolor{Gray}
\hline
l,m & 0 & 1 & 2 & 3 & 4 & 5 & 6 & 7 & 8 & 9\\ 
\hline
1 & 1 ($ s_{[4]}$) & 0 & 3 ($ s_{[2, 1, 1]}$) & 0 & - & - & - & - & - & -\\ 
2 & 1 ($ s_{[4]}$) & 0 & 0 & 0 & 2 ($ s_{[2, 2]}$) & 3 ($ s_{[2, 1, 1]}$) & - & - & - & -\\ 
3 & 1 ($ s_{[4]}$) & 0 & 0 & 0 & 6 ($ s_{[4]}$ + $s_{[3, 1]}$ + $s_{[2, 2]}$) & 0 & 11 ($ 2s_{[3, 1]}$ + $s_{[2, 2]}$ + $s_{[2, 1, 1]}$) & 0 & - & -\\ 
\hline
\end{tabular}

 }

\smallskip

 5 hairs 

\scalebox{ 1 }{
 \begin{tabular}{|g|D|D|D|D|D|M|D|D|D|D|}
 \rowcolor{Gray}
\hline
l,m & 0 & 1 & 2 & 3 & 4 & 5 & 6 & 7 & 8 & 9\\ 
\hline
1 & 1 ($ s_{[5]}$) & 0 & 6 ($ s_{[3, 1, 1]}$) & 0 & 1 ($ s_{[1, 1, 1, 1, 1]}$) & - & - & - & - & -\\ 
2 & 1 ($ s_{[5]}$) & 0 & 0 & 0 & 10 ($ s_{[3, 2]}$ + $s_{[2, 2, 1]}$) & 15 ($ s_{[3, 1, 1]}$ + $s_{[2, 2, 1]}$ + $s_{[2, 1, 1, 1]}$) & 0 & - & - & -\\ 
3 & 1  & 0  & 0  & 0  & 10  & 0  & 41  & 0  & 9  & -\\ 
\hline
\end{tabular}

 }
  \caption{Homology dimensions for $\FG_n$ with $n$ even.}
  \label{fig:forested_odd}
\end{figure}

\appendix

\section{An auxiliary Proposition for the forested graph complexes}
\label{app:FGC dc bottom proof}

In this section we shall prove the second part of Proposition \ref{prop:FGC dc bottom} above, namely the following statement:

\begin{prop}\label{prop:FGC dc bottom app}
For each $n$ and $g\geq 1$ the homology of $\FGC_{n}^{bl,g\lp}$ with respect to the edge contraction differential $d_c$ is concentrated in excess zero, corresponding to graphs all of whose vertices are trivalent.
\end{prop}
Note that the Proposition would immediately follow from the Koszulness of the commutative (or Lie) operad if in the definition of $\FGC_{n}$ we would not restrict to bridgeless graphs.
To show the more complicated version above, we shall use a "bridgeless" variant of the Koszulness of $\Com$.

To this end we introduce some notation.
We consider the cyclic bar construction of the commutative operad $\bCom$.
Elements of $\bCom((r))$ can be seen as linear combinations of non-rooted trees with $r$ numbered leaves.
The differential acts by edge contraction.
We fix an integer $N$ and a partition 
\[
  \underline p = (p_1,\dots,p_N)
\]
of $r$, i.e., $r=p_1+\cdots +p_N$.
The partition can be considered as a coloring of the leaves of trees in $\bCom((r))$, with the first $p_1$ leaves of color 1, the following $p_2$ leaves of color 2 etc.

\begin{defi}
  Let $T$ be an unrooted tree with $r$ leaves and let $\underline p$ be a partition of $r$, thought of as a coloring of the $r$ leaves.
  Then we say that an edge $e$ in $T$ is a $\underline p$-bridge, if removing the edge disconnects the tree $T$ in components $T_1$ and $T_2$ whose sets of leaf colors are disjoint.
  We say that $T$ is $\underline p$-\emph{bridgeless} if $T$ does not have a $\underline p$-bridge.
\end{defi}

The property of being $\underline p$-bridgeless is stable under edge contraction. We can define the subcomplex 
\[
V_{r, \underline p} \subset \bCom((r))
\]
spanned by the $\underline p$-bridgeless trees.

\begin{lemma}\label{lem:bridgeless Koszul}
Fix integers $r\geq 3$ and $N$ with $r\geq N+1$ and a partition $\underline p=(p_1,\dots,p_N)$ of $r$ with $p_j\geq 1$ for $j=1,\dots,N$. 
Then $H^\bullet(V_{r, \underline p})$ is concentrated in the bottom degree $\bullet = 3-r$, corresponding to trivalent trees.
\end{lemma}
Note that if $r=N$ the statement of the Lemma is false: In this case each leaf has its own color and every internal edge of a tree is a $\underline p$-bridge.
This means that $V_{r, \underline p}$ is 1-dimensional, spanned by the unique tree without edges, of degree -1.

\begin{proof}
We perform an induction on $r$. For $r=3$ the Lemma is trivially true.
We then perform the induction step $r-1\to r$, assuming that the Lemma is true up to arity $r-1$, and for all partitions.

By assumption there are 2 leaves of the same color, and without loss of generality we can assume that leaves 1 and 2 are of color 1.

We endow $V_{r, \underline p}$ with a filtration 
\[
0 = \mF^{-1} V_{r, \underline p} \subset \mF^0 V_{r, \underline p} \subset \cdots
\]
such that $\mF^q V_{r, \underline p}$ is spanned by trees $T$ that have at most $q$ edges \emph{not} contained in the unique path connecting leaves 1 and 2.
We claim that the $E^2$ page of the spectral sequence associated to our filtration is concentrated in degree $3-r$, thus showing the Lemma.

We look at the first page of the associated spectral sequence, that is, the associated graded. It can be identified with $V_{r, \underline p}$ with differential $d_0$ contracting (only) edges along the unique path connecting leaves 1 and 2, leaving the other edges untouched.

Let us call that unique path in a tree $T$ the \emph{backbone} of the tree.
Let us call the individual subtrees $T_1, T_2,\dots$ connected to the backbone the ribs of the graph, see the following picture:
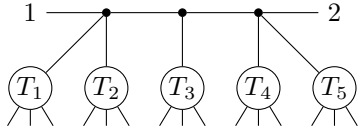
\[
\begin{tikzpicture}
\node (v1) at (0,0) {$1$};
\node (v2) at (4,0) {$2$};
\node[int] (w1) at (1,0) {};
\node[int] (w2) at (2,0) {};
\node[int] (w3) at (3,0) {};
\node[ext] (T0) at (0,-1) {$T_1$};
\node[ext] (T1) at (1,-1) {$T_2$};
\node[ext] (T2) at (2,-1) {$T_3$};
\node[ext] (T3) at (3,-1) {$T_4$};
\node[ext] (T4) at (4,-1) {$T_5$};
\draw (w1) edge (T1) edge (v1) edge (w2) 
(w3) edge (w2) edge (v2) edge (T3) 
(T1) edge +(-.3,-.5) edge +(0,-.5) edge +(.3,-.5)
(T2) edge +(-.3,-.5) edge +(0,-.5) edge +(.3,-.5) edge (w2)
(T3) edge +(-.3,-.5) edge +(0,-.5) edge +(.3,-.5)
(T0) edge +(-.3,-.5) edge +(0,-.5) edge +(.3,-.5) edge (w1) 
(T4) edge +(-.3,-.5) edge +(0,-.5) edge +(.3,-.5) edge (w3)
;
\end{tikzpicture}
\]
Note that the backbone edges can never be $\underline p$-bridges, because the backbone is the unique path connecting leaves 1 and 2 of like color.
The differential $d_0$ leaves the ribs the same, and just contracts backbone edges. The resulting complex is isomorphic to a direct summand of the (associative) bar complex of a free commutative algebra, with generators corresponding to the ribs.
It is well known that the homology of the associative bar complex of a free commutative algebra is a cofree cocommutative coalgebra, with degree shifted generators.
In our setting, this means that the homology $H(V_{r, \underline p},d_0)$ 
is freely generated by linear combinations of trees all of whose backbone vertices are trivalent, and the ribs are attached anti-symmetrically:
\[
  \sum_{\sigma \in S_3}
  (-1)^\sigma
    \begin{tikzpicture}
    \node (v1) at (0,0) {$1$};
    \node (v2) at (4,0) {$2$};
    \node[int] (w1) at (1,0) {};
    \node[int] (w2) at (2,0) {};
    \node[int] (w3) at (3,0) {};
    \node[ext] (T1) at (1,-1) {$\scriptstyle T_{\sigma(1)}$};
    \node[ext] (T2) at (2,-1) {$\scriptstyle T_{\sigma(2)}$};
    \node[ext] (T3) at (3,-1) {$\scriptstyle T_{\sigma(3)}$};
    \draw (w1) edge (T1) edge (v1) edge (w2) 
    (w3) edge (w2) edge (v2) edge (T3) 
    (T1) edge +(-.3,-.7) edge +(0,-.7) edge +(.3,-.7)
    (T2) edge +(-.3,-.7) edge +(0,-.7) edge +(.3,-.7) edge (w2)
    (T3) edge +(-.3,-.7) edge +(0,-.7) edge +(.3,-.7)
    ;
    \end{tikzpicture}
\]

The differential $d_1$ on the next page of our spectral sequence acts by contracting a non-root edge inside a rib.
The resulting complex decomposes into tensor products of complexes for individual ribs, which can in turn be seen to be isomorphic to $V_{r', \underline p'}$ for some $r'<r$. Applying the induction hypothesis hence shows that the homology of the $E^1$-page of our spectral sequence can be represented by trivalent trees.
\end{proof}

\begin{proof}[Proof of Proposition \ref{prop:FGC dc bottom app}]

By definition, the forested graph complex with bridges is the (appropriately degree shifted) Feynman transform of the cyclic bar construction of $\Com$. 
We can hence consider a forested graph as a "meta"-graph whose vertices are decorated by elements of the cyclic bar construction of $\Com$.
The shape of the meta-graph is unaltered by the differential $d_c$ acting only on the decorations at the meta-vertices.
Hence the forested graph complex with bridges decomposes into direct summands of complexes of the form 
\[
\otimes_G^\Lambda \bCom  \cong \left(\Lambda\otimes \otimes_{v\in V G} (\bCom)((r_v)) \right)_{Aut(G)} 
\]
for meta-graphs $G$, and $\Lambda$ a one-dimensional representation of $Aut(G)$ taking care of signs and degree shifts.

To pass to the bridgeless version, let us consider a metagraph $G$ and one meta-vertex $v$ in $G$, with incident half-edges $h_1,\dots,h_r$. Deleting $v$ the graph decomposes $G$ into a union of say $N$ connected components. We color the half-edge $h_j$ by the connected component they belong to (after deleting $v$).
The bridgeless forested graph complex then decomposes into summand of the form
\[
\left(\Lambda\otimes \otimes_{v\in VG} V_{r_v,\underline p_v} \right)_{Aut(G)} ,
\]
for bridgeless metagraphs $G$, with $\underline p_v$ the partition corresponding to the above coloring of the half-edges at $v$. 
To apply the Lemma above and conclude, we hence just have to check that the coloring $\underline p_v$ is such that at least two half-edges have the same color, for each vertex $v$.
Assume there was some $v$ such that each of the $r$ half-edges has its own color, i.e., $G$ decomposes into $r$ connected components after deleting $v$.
Then necessarily each of the half-edges must be an external leg, since if one belonged to an internal edge, that edge would be a bridge.
But not each half-edge can be an external leg, since otherwise $G$ would be of loop order $0$, contradicting our assumption $g\geq 1$. 
Hence Lemma \ref{lem:bridgeless Koszul} is applicable and the Proposition follows.
\end{proof}

\printbibliography

\end{document}